\documentclass[english]{article}

\date{} 
\usepackage[margin=1in,footskip=0.25in]{geometry} 

\usepackage{lipsum}
\usepackage{amsfonts}
\usepackage{graphicx}
\usepackage{epstopdf}
\usepackage{calc}

\usepackage{ifpdf}
\usepackage{amsmath,wasysym, tikz-cd, soul}
\usepackage{subcaption}

\definecolor{darkgreen}{RGB}{0, 102, 51}
\definecolor{darkblue}{RGB}{0, 0, 153}
\usepackage[colorlinks, linkcolor = darkblue, citecolor = darkgreen, bookmarks=false]{hyperref}
\hypersetup{final = true}

\usepackage{amssymb}
\usepackage{amsthm}
\usepackage[mathscr]{euscript}
\usepackage{amscd}
\usepackage{extarrows}

\usepackage{thmtools}
\usepackage{thm-restate}

\usepackage{amsopn}

\usepackage{algorithm}
\usepackage{algpseudocode}

\usepackage{graphicx}
\usepackage{dcolumn}
\usepackage{enumitem}
\usepackage{mathtools}
\usepackage{array}
\graphicspath{ {Figures/} }
\usepackage{textcomp}
\usepackage[margin=1cm]{caption} 
\usepackage[absolute,overlay]{textpos}
\usepackage{setspace}

\usepackage[title]{appendix} 

\ifpdf
  \DeclareGraphicsExtensions{.eps,.pdf,.png,.jpg}
\else
  \DeclareGraphicsExtensions{.eps}
\fi

\usepackage{enumitem}
\setlist[enumerate]{leftmargin=.5in}
\setlist[itemize]{leftmargin=.5in}

\newtheorem{Lemma}{Lemma}
\numberwithin{Lemma}{section}
\newtheorem{Theorem}{Theorem}

\DeclareMathOperator{\im}{\textrm{im}} 

\newcommand{\field}{\mathbb}
\newcommand {\sheaf}{\mathscr}
\newcommand {\cosheaf}{\mathscr}
\newcommand{\FaceRelation}{\unlhd} 

\newcommand{\CellComplex}{} 
\newcommand{\total}{\textrm{Tot}}
\newcommand{\barcode}{\textit{barcode}} 
\newcommand{\incv}{\iota}
\newcommand{\ince}{\kappa}
\newcounter{example}[section]
\newenvironment{example}[1][]
{\refstepcounter{example}\par\medskip
   \textbf{Example~\theexample. #1} \rmfamily}{\medskip}

\newcommand{\wmodule}{\mathbb{O}} 
\newcommand{\wmap}{\omega} 

\title{Persistence by Parts: Multiscale Feature Detection via Distributed Persistent Homology}
\author{Hee Rhang Yoon\thanks{Department of Electrical \& Systems Engineering, University of Pennsylvania, Philadelphia, PA   (yoonhee@seas.upenn.edu)} \hspace{0.05cm} and Robert W. Ghrist\thanks{Department of Mathematics and Electrical \& Systems Engineering, University of Pennsylvania, Philadelphia, PA (ghrist@math.upenn.edu)}}

\begin{document}

\maketitle

\begin{abstract}
  A method is presented for the distributed computation of persistent homology, based on an extension of the generalized Mayer-Vietoris principle to filtered spaces. Cellular cosheaves and spectral sequences are used to compute global persistent homology based on local computations indexed by a scalar field. These techniques permit computation localized not merely by geography, but by other features of data points, such as density. As an example of the latter, the construction is used in the multi-scale analysis of point clouds to detect features of varying sizes that are overlooked by standard persistent homology.
\end{abstract}

\section{Introduction}
\label{intro}

\subsection{Cosheaves and computational persistence}
\label{sec:coco}

Persistence has emerged as a central principle in Topological Data Analysis (TDA). As its core, persistence exploits the functoriality of homology to extract robust features from point cloud data. Given a filtered cell complex arising from such data, its homology is a persistence module -- a sequence of vector spaces and linear transformations. The representation theory of such objects leads to an unambiguous decomposition into indecomposables (persistent homology classes) \cite{CarlssonTopology09,Carlssonshape13,CSZigzag10} which, thanks to the Stability Theorem for persistent homology \cite{CEHStability07,CSG+structure12}, is robust with respect to perturbations of the original data points. For details, see \S\ref{sec:PH}.

There is no inherent obstruction to computing homology from local data: Mayer-Vietoris and spectral sequences lead the way. The story is more complicated for persistent homology, due to the algebra of persistence modules. The approach of this paper is to use theory of cellular cosheaves to store and relate local homology information over the complex. This theory, like the algebraically dual theory of cellular sheaves \cite{CurrySheaves14}, is ideal for local-to-global operations.

The main result of this paper is an application of spectral sequences and the Mayer-Vietoris principle to compute persistent homology by breaking a filtered complex into pieces based on a scalar field on the point cloud. This has utility beyond the obvious idea of breaking up a complex into parts based on geographic proximity.

In its typical interpretation, homology classes which persist over large parameter intervals are the statistically significant features -- those which are not artefacts of noisy sampling. For \v{C}ech complexes of manifolds sampled with sufficiently high density, a single (practically unobtainable)  homology computation yields truth \cite{NSWFinding08}. For the (more realistic) case of a filtered Vietoris-Rips complex, persistent homology is a good first approximation to truth. Questions of a more epistemological bent ({\em ``Are these really the important features?''}) are evident. Recent work of Berry and Sauer argues that persistent homology can erroneously label important small features as non-persistent, and they propose a {\em Continuous $k$-Nearest Neighbors} graph to estimate the geometry of a discretely sampled manifold \cite{BSConsistent19}. They prove that this method captures multiscale connectivity data and conjecture that it works in higher homologies.

As an application of our results, we show how to use a local density as the distribution parameter in order to retrieve persistent homology classes that are weighted according to the density of sampling. The idea is that when faced with geometrically small but tightly-sampled homology classes which may be of significance, one can compute persistent homology using distance as the (usual) parameter, and sampling density as the distribution field: see \S\ref{sec:multi}. The oft-expressed desire of doing multiparameter persistence on both distance and density \cite{CarlssonTopology09,GhristBarcodes08} is full of algebraic challenges \cite{CZtheory09,CSZComputing10}; in a sense, this paper gives a novel approach by separating out one parameter as a scalar field for distributed computation.

\subsection{Related and supporting work}

Algorithms for computing persistent homology now comprise a rich and intricate literature. The original algorithm for computing persistent homology \cite{ELZTopological02,EHComputational10} computes persistence pairs by reducing the boundary operator. Since then, variations of the original algorithm have been developed to improve computation \cite{OPT+roadmap}. In particular, parallelized and distributed algorithms have been proposed in order to improve memory usage and computation time. The spectral sequence algorithm from \cite{EHComputational10} reduces blocks of matrices at each phase, which results in persistence pairs of particular lengths. In \cite{BKRClear14}, the authors incorporate optimization techniques and construct a distributed algorithm of the spectral sequence algorithm in both shared and distributed memory.

The above mentioned algorithms distribute data with respect to ranges of filtration values. In \cite{LMParallel15}, the authors provide a distributed computation algorithm in which data is distributed with respect to spatial decomposition of the domain. They build a Mayer-Vietoris blowup complex, and its boundary matrix is reduced by reducing submatrices in parallel. Our work shares a similar philosophy, as we distribute data according to spatial decomposition of the domain and we operate on the foundation of Mayer-Vietoris principle. Instead of using the geometric construction of Mayer-Vietoris blowup complex, we use the algebraic construction of cosheaf homology to combine local information. Our approach adapts the use of cosheaf homology to compute homology from subspaces \cite{CGNDiscrete13}, with cosheaf morphisms and spectral sequences to take the filtration into account.

This work has its origins in the Ph.D. dissertation of the first author \cite{YoonCellular18}. After preparing this article, the preprint of Casas appeared \cite{CasasDistributing19}, which, influenced by \cite{YoonCellular18}, builds on and extends the results. In particular, Algorithm 2 of \cite{CasasDistributing19} recapitulates \S 4.2.3 of \cite{YoonCellular18}. Casas greatly extends the results of that thesis and this paper by not limiting the nerve of the distribution cover to 1-d; however, in the restricted case considered here, Casas' diagram chase is exactly that of \cite{YoonCellular18} and the present work.

\subsection{Problem statement and contributions}
We address the following question.

\smallskip
\textit{Given a point cloud $P$ (a finite subset of Euclidean $\mathbb{R}^m$), compute the persistence module of $P$ from local persistence modules subordinate to a cover of $P$}
\smallskip

Let $\mathbb{V}$ denote the persistence module obtained from $P$ --- a finite sequence of vector spaces and linear maps that encodes topological information about Vietoris-Rips filtration associated to $P$. We use cellular cosheaf homology to assemble the relevant information gathered from subsets of $P$ subordinate to a cover. Morphisms of cellular cosheaves then allow us to incorporate persistence. We use spectral sequences to discover a hidden map among cosheaf homologies. Our main result, as stated in the following theorem, is a distributed construction of a persistence module $\mathbb{V}_{\Psi}$ that is isomorphic to the persistence module $\mathbb{V}$ of interest.

\begin{Theorem}
The local $\mathbb{V}_{\Psi}$ and global $\mathbb{V}$ persistence modules are isomorphic.
\end{Theorem}

We argue that this distributed computation can be used to filter and annotate persistent homology based on meta-data associated to the point-cloud, using such to build a cover. We illustrate this idea in the case where the meta-data is a sampling density estimate, yielding a method for computing multiscale persistent homology, identifying significant features that are overlooked by standard persistent homology methods.

Section \ref{sec:background} contains a summary of persistent homology and cellular sheaf theory. We review a distributed computation method for homology using Leray cellular cosheaves in \S\ref{sec:DistributedHomology}. In \S\ref{sec:DistributedPH}, we construct a general distributed persistent homology computation mechanism subordinate to a cover with at most pairwise overlaps. Finally, in \S\ref{sec:multi}, we indicate the utility of this distributed computation in multiscale persistent homology, using our methods to identify persistent homology classes relative to a sampling density.

\section{Background}
\label{sec:background}

Throughout this paper, we assume that every homology is computed with coefficients in a field $\field{K}$.

\subsection{Persistent Homology}
\label{sec:PH}
Given a point cloud, one can interrogate its global structure via persistent homology. Let $P$ be a finite collection of points in Euclidean $\mathbb{R}^m$. The \textbf{Vietoris-Rips complex} $\mathscr{R}^{\epsilon}$ is the simplicial complex whose $k$-simplices correspond to $(k+1)$-tuple of points from $P$ that have pairwise distance $\leq \epsilon$. For brevity, we use the term ``Rips complex" to refer to the Vietoris-Rips complex.

Assume that $( \mathscr{R}^{i} )_{i=1}^N$ is a sequence of Rips complexes over $P$ for increasing parameter values $( \epsilon_i ) _{i=1}^N$, with inclusion maps between each pair of Rips complexes
\[\mathscr{R}^1 \hookrightarrow \mathscr{R}^{2} \hookrightarrow \dots \hookrightarrow \mathscr{R}^{N}.\]
By applying the homology functor, one obtains the following sequence of vector spaces
\begin{equation}
\label{OriginalPM}
\mathbb{V}: H_{\bullet}(\mathscr{R}^{1}) \to H_{\bullet}(\mathscr{R}^{2}) \to \dots \to H_{\bullet}(\mathscr{R}^{N}).
\end{equation}
The above sequence is an instance of a \textbf{persistence module}. For the purpose of this paper, it suffices to consider a persistence module as a finite sequence of vector spaces and linear maps between them.

A \textbf{morphism of persistence modules} $\alpha:\mathbb{V} \to \mathbb{W}$ is a collection of linear maps $\alpha_i : V_i \to W_i$ such that the following diagram commutes.
\[
\begin{tikzcd}
V_1  \arrow[d, "\alpha_1"] \arrow[r]& V_2 \arrow[d, "\alpha_2"] \arrow[r]& \cdots \arrow[r] & V_n \arrow[d,"\alpha_n"] \\
W_1 \arrow[r] & W_2 \arrow[r] & \cdots \arrow[r] & W_n
\end{tikzcd}
\]
When all $\alpha_i$'s are isomorphisms, then $\alpha$ is an isomorphism of persistence modules.

By the Structure Theorem \cite{ZCComputing05}, the persistence module from Equation (\ref{OriginalPM}) decomposes uniquely as
\[ \mathbb{V} \cong \bigoplus\limits_{l=1}^N \mathbb{I}(b_l, d_l),\]
where each $\mathbb{I}(b,d)$ is a simpler persistence module of the form
\[ \mathbb{I}(b,d): 0 \to \dots 0 \to \field{K} \xrightarrow{1} \field{K} \xrightarrow{1}  \dots \xrightarrow{1} \field{K} \rightarrow 0 \to \dots \to 0.\]
The $b$ and $d$ each indicates the first and last index of $\field{K}$. Each $\mathbb{I}(b,d)$ represents a homological feature with birth time $b$ and death time $d$. One can visualize such birth and death times of $\mathbb{I}(b,d)$ using a \textbf{barcode}. Given a persistence module $\mathbb{V}$, a barcode diagram, $\barcode(\mathbb{V})$, is a collection of bars that correspond to the intervals $(b,d)$ obtained from the decomposition of $\mathbb{V}$. In simple settings, long bars of $\barcode(V)$ capture significant homological features; shorter bars may be due to noise.

\subsection{Cellular Cosheaves}
\label{sec:CellularCosheaves}

A cellular cosheaf is a certain assignment of algebraic structure to a cell complex \cite{CurrySheaves14}.  Given a cell complex $\CellComplex{X}$, there is a face poset category whose objects are the cells of $\CellComplex{X}$ and whose morphisms $\tau\to\sigma$ correspond to the face relation $\tau \FaceRelation \sigma$ (that is, $\tau \subset \bar{\sigma}$).

Following \cite{ShepardCellular85,CurrySheaves14}, a \textbf{cellular cosheaf} $\sheaf{F}$ on a cell complex $\CellComplex{X}$ with values in category $D$ is a contravariant functor from the associated face poset category to $D$. In other words, $\sheaf{F}$ assigns to each cell $\sigma$ of $\CellComplex{X}$ an object $\sheaf{F}(\sigma)$ in $D$, and to each face relation $\tau \FaceRelation \sigma$ a morphism $\sheaf{F}(\tau \FaceRelation \sigma):\sheaf{F}(\sigma) \to \sheaf{F}(\tau)$ such that
\begin{itemize}
\item $\sheaf{F}(\tau \FaceRelation \tau) : \sheaf{F}(\tau) \to \sheaf{F}(\tau)$ is the identity, and
\item if $\rho \FaceRelation \tau \FaceRelation \sigma$, then $\sheaf{F}( \rho \FaceRelation \sigma ) =\sheaf{F}(\rho \FaceRelation \tau) \circ \sheaf{F}(\tau \FaceRelation \sigma).$
\end{itemize}

A cellular cosheaf $\cosheaf{F}$ over a compact cell complex $X$ has a well-defined homology associated to its chain complex
\begin{equation}
\label{chaincomplex}
C_{\bullet} \cosheaf{F} : \quad \cdots \xrightarrow{\partial_3} \bigoplus \limits_{\textrm{dim }\sigma=2 } \cosheaf{F}(\sigma) \xrightarrow{\partial_2}  \bigoplus \limits_{\textrm{dim } \sigma =1}\cosheaf{F}(\sigma) \xrightarrow{\partial_1} \bigoplus \limits_{\textrm{dim } \sigma =0} \cosheaf{F}(\sigma) \xrightarrow{\partial_0} 0,
\end{equation}
where $C_n(\CellComplex{X}, \cosheaf{F})$ is the direct sum of the data over the $n$-cells of $X$. The boundary map $\partial_n : C_n(\CellComplex{X}, \cosheaf{F}) \to C_{n-1}(\CellComplex{X}, \cosheaf{F})$ is defined in the familiar manner as
\[
 \partial_n = \sum\limits_{\tau \FaceRelation \sigma} [\tau : \sigma] \cosheaf{F}(\tau \unlhd \sigma),
\]
where $[\tau:\sigma]$ is the incidence number.

The $n^{\textrm{th}}$ \textbf{cosheaf homology} of $\cosheaf{F}$ is the homology of this chain complex $H_n(C_{\bullet} \cosheaf{F}) = \ker \partial_n / \im \partial_{n+1}$.

Cosheaf homologies reflect global structures from locally encoded data. When comparing local data, cosheaf morphisms allow one to extract global changes from the local changes. Following \cite{BredonSheaf97},
a \textbf{cosheaf morphism} $\phi : \sheaf{F} \to \sheaf{F}'$ between cosheaves $\sheaf{F}$ and $\sheaf{F}'$ on $\CellComplex{X}$ is a collection of morphisms $\phi|_{\sigma} :\sheaf{F}(\sigma) \to \sheaf{F'}(\sigma)$ indexed by cells $\sigma \in X$ such that the following diagram commutes
\[
\begin{tikzcd}
\sheaf{F}(\sigma) \arrow[d,"\sheaf{F}(\tau \unlhd \sigma)" , swap] \arrow[r,"\phi|_{\sigma}"] & \sheaf{F}'(\sigma) \arrow[d,"\sheaf{F}'( \tau \unlhd \sigma)" ]\\
\sheaf{F}(\tau) \arrow[r,"\phi|_{\tau}"] & \sheaf{F}'(\tau)
\end{tikzcd}
\]
for every face relation $\tau \FaceRelation \sigma$.

Thus, a cosheaf morphism $\phi : \sheaf{F} \to \sheaf{F}'$ is a natural transformation from the functor $\sheaf{F}$ to $\sheaf{F}'$. As such,
any cosheaf morphism $\phi : \cosheaf{F} \to \cosheaf{F}'$ induces morphisms on homology:
\[
	H_n(\phi) : H_n(C_{\bullet}\cosheaf{F}) \to H_n(C_{\bullet} \cosheaf{F}') .
\]

In the special case of a cell complex $\CellComplex{X}$ that is homeomorphic to a closed interval, a cellular cosheaf $\cosheaf{F}$ on $\CellComplex{X}$ can be interpreted as a generalized persistence module, or \textbf{zigzag module} \cite{CSZigzag10}, of the form
\[
	V_1 \leftarrow V_2 \rightarrow V_3 \leftarrow \cdots \rightarrow V_m.
\]
By Gabriel's Theorem \cite{GabrielUnzerlegbare72}, such a cosheaf $\cosheaf{F}$ can be decomposed into a direct sum of \textbf{indecomposable cosheaves}, of the form
\[
	\mathscr{I} : 0 \leftarrow \cdots 0 \leftrightarrow \field{K} \leftrightarrow \cdots \leftrightarrow \field{K} \leftrightarrow 0 \cdots \rightarrow 0 ,
\]
as in Figure~\ref{fig:DecomposedCosheaf}.
\begin{figure}[h!]
\centering
\includegraphics[scale=0.25]{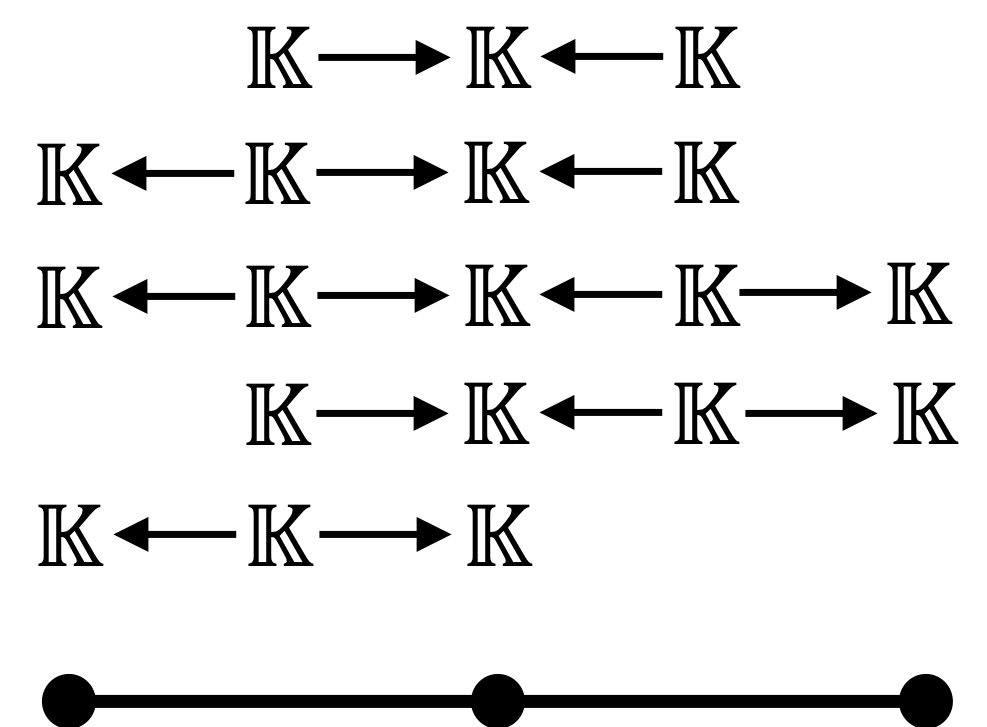}
\caption{Direct sum of indecomposable cosheaves}
\label{fig:DecomposedCosheaf}
\end{figure}

Note that there are four types of indecomposable cosheaves possible: cosheaves $\sheaf{I}_{[-]}$ whose left- and rightmost supports are 0-cells, cosheaves $\sheaf{I}_{]-[}$ whose left and right most supports are 1-cells, cosheaves $\sheaf{I}_{[-[}$ with leftmost support a 0-cell and rightmost support a 1-cell, and the reversed $\sheaf{I}_{]-]}$.

\begin{Lemma}[\cite{CurrySheaves14}]
\label{IndecomposableCosheaf}
The indecomposable cosheaves $\cosheaf{I}$ satisfy
\[
	H_0(C_{\bullet}\cosheaf{I}_{[-]}) = \field{K}, \quad H_1(C_{\bullet}\cosheaf{I}_{]-[})=\field{K}, \quad H_i(C_{\bullet}\cosheaf{I}_{[-[})=H_i(C_{\bullet}\cosheaf{I}_{]-]})=0.
\]
\end{Lemma}

Thus, if cosheaf $\cosheaf{F}$ can be decomposed as $\cosheaf{F} \cong \oplus \cosheaf{I}$, then $H_i(C_{\bullet} \cosheaf{F}) \cong \oplus H_i(C_{\bullet} \cosheaf{I})$. Thus, $\dim H_0(C_{\bullet}\sheaf{F})$ counts the number of indecomposable cosheaves $\cosheaf{I}_{[-]}$ in the decomposition, whereas $\dim H_1(C_{\bullet}\sheaf{F})$ counts the number of indecomposable cosheaves $\cosheaf{I}_{]-[}$ in the decomposition. Lemma \ref{IndecomposableCosheaf} is the key to enriching the persistent homology barcodes in \S\ref{sec:multi}.

\section{Distributed computation of homology}
\label{sec:DistributedHomology}

Our goal is to compute the persistence module
\[
	\mathbb{V}: H_{\bullet} (\mathscr{R}^1) \to H_{\bullet} (\mathscr{R}^2) \to \cdots \to H_{\bullet} (\mathscr{R}^N)
\]
in a distributed manner. To commence, we recall the local nature of homology computations,
drawing on the classic results of Mayer-Vietoris, Leray \cite{BredonSheaf97}, and Serre \cite{McClearyUsers01}, in language of sheaf cohomology \cite{CGNDiscrete13}. The following adapts the classic constructions from \cite{CGNDiscrete13} to analyzing point cloud data.

For $P$ a point cloud, denote by $\mathscr{R}^{\epsilon}$ the Rips complex built on $P$ for parameter $\epsilon$.
Let $\mathscr{V}$ be a finite cover of $P$ and $N_{\mathscr{V}}$ the resulting nerve complex.\footnote{The simplicial complex that indexes intersections of cover elements.}
To each simplex $\sigma \in N_{\mathscr{V}}$ is then associated $\mathscr{R}_{\sigma}^{\epsilon}$, the Rips complex built on the subset of points of $P$ indexed by $\sigma$ at proximity parameter $\epsilon$.

We will refer to the collection $\coprod\limits_{\dim \sigma=n } \mathscr{R}^{\epsilon}_{\sigma}$ as the ``Rips complexes over the $n$-simplices of $N_{\mathscr{V}}$", referring to the entire collection $\coprod\limits_{\sigma \in N_{\mathscr{V}}} \mathscr{R}^{\epsilon}_{\sigma}$ as the \textbf{Rips system} of the nerve.

Define a cosheaf $\cosheaf{F}^{\epsilon}_n$ on $N_{\mathscr{V}}$ as the following. For each $\sigma \in N_{\mathscr{V}}$, let $\cosheaf{F}_{n}^{\epsilon} (\sigma)= H_n( \mathscr{R}^{\epsilon}_{\sigma} )$. For $\tau \unlhd \sigma$, let $\cosheaf{F}_n^{\epsilon}(\tau \unlhd \sigma)$ be the map induced by inclusion $\mathscr{R}^{\epsilon}_{{\sigma}} \hookrightarrow \mathscr{R}^{\epsilon}_{\tau}$.

\begin{restatable}{Lemma}{discreteCGNlemma}
\label{DiscreteCGNBound}
Let $P$ be a point cloud with finite cover $\mathscr{V}$ having one-dimensional nerve $N_{\mathscr{V}}$. There exists a constant $\epsilon_*$ such that
\begin{equation}
\label{DistributedComputationK}
H_n(\mathscr{R}^{\epsilon}) \cong H_0(C_{\bullet} \cosheaf{F}^{\epsilon}_n) \oplus H_1(C_{\bullet} \cosheaf{F}^{\epsilon}_{n-1})
\end{equation}
for every $0<\epsilon<\epsilon_*$.
\end{restatable}
\begin{proof}
Theorem 5.7 in \cite{CGNDiscrete13} can be used to obtain the above isomorphism when the Rips system covers the full Rips complex: the technical work is in showing that this happens for sufficiently small $\epsilon$: see Appendix \ref{A:DiscreteCGNBound} for details.
\end{proof}
The following two examples illustrate the difference between $H_0(C_{\bullet}\cosheaf{F}_1^{\epsilon})$ and $H_1(C_{\bullet}\cosheaf{F}_0^{\epsilon})$.
\begin{example}
\label{SmallE}
Let $P \subset \mathbb{R}^2$ be a point cloud covered by three sets $V_1, V_2, V_3$ with nerve an interval as illustrated in Figure \ref{fig:AllPoints}.
\begin{figure}[h!]
\centering
\includegraphics[scale=0.4]{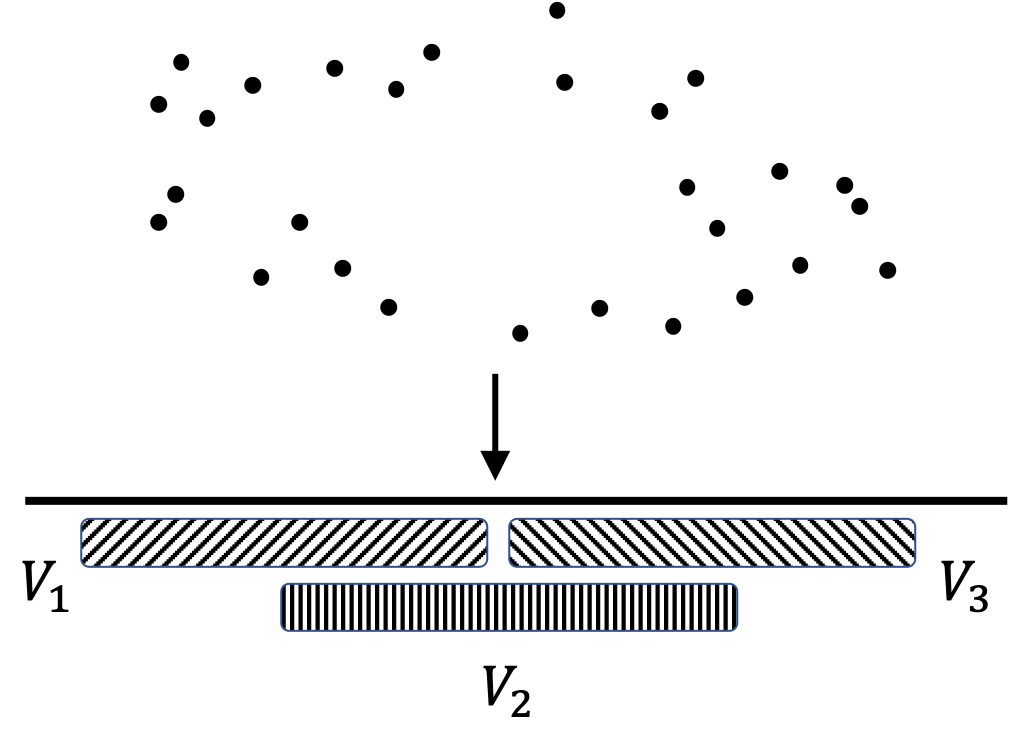}
\caption[Point cloud $P$ and covering by inverse image of projected intervals.]{Point cloud $P$ and covering by inverse image of projected intervals.}
\label{fig:AllPoints}
\end{figure}

Consider $H_1(\mathscr{R}^{\epsilon})$ for some parameter $\epsilon$. The Rips complex $\mathscr{R}^{\epsilon}$ and the Rips system over the nerve $N_{\mathscr{V}}$ are illustrated in Figures \ref{fig:RipsComplexEx1} and \ref{fig:RipsSystemEx1}. Let $v_1, v_2, v_3$ denote the vertices of $N_{\mathscr{V}}$ that correspond to the cover sets $V_1, V_2, V_3$. Let $e_{12}$ and $e_{23}$ denote the edges of $N_{\mathscr{V}}$ that correspond to $V_1 \cap V_2$ and $V_2 \cap V_3$.

\begin{figure}[H]
\begin{subfigure}{0.3\textwidth}
\vfill
\includegraphics[scale=0.3]{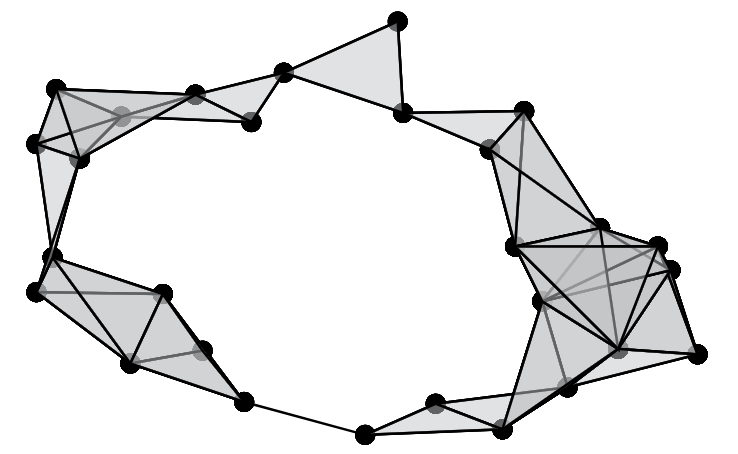}
\vfill
\caption{Rips complex $\mathscr{R}^{\epsilon}$.}\label{fig:RipsComplexEx1}
\end{subfigure}\hspace{18mm}%
\begin{subfigure}{0.55\textwidth}
\includegraphics[scale=0.2]{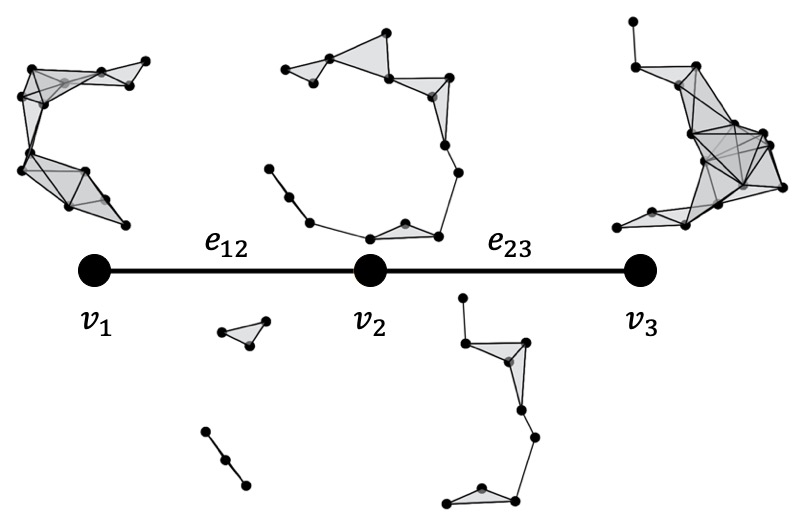}
\caption{Rips system over the nerve $N_{\mathscr{V}}$. }\label{fig:RipsSystemEx1}
\end{subfigure}
\caption{Rips complex and the associated Rips system.}\label{fig:RipsSystem80}
\end{figure}

The two relevant cosheaves, $\cosheaf{F}_0^{\epsilon}$ and $\cosheaf{F}_1^{\epsilon}$, are illustrated in Figures \ref{fig:F12} and \ref{fig:F02}. The maps $\cosheaf{F}^{\epsilon}_0(v_1 \FaceRelation e_{12})$ and  $\cosheaf{F}^{\epsilon}_0(v_2 \FaceRelation e_{12})$ are represented by the matrix $\begin{bmatrix}
1  & 1 \end{bmatrix} $. All other maps are identity maps.

\begin{figure}[H]
\hfill
\begin{subfigure}{0.35\textwidth}
\includegraphics[width=1\textwidth]{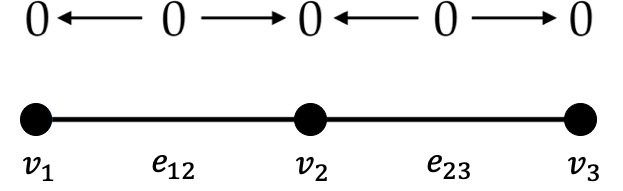}
\caption{Cosheaf $\cosheaf{F}_1^{\epsilon}$.}\label{fig:F12}
\end{subfigure}%
\hfill
\begin{subfigure}{0.35\textwidth}
\includegraphics[width=1\textwidth]{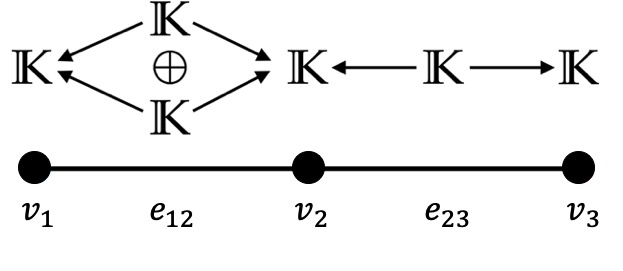}
\caption{Cosheaf $\cosheaf{F}_0^{\epsilon}$.}\label{fig:F02}
\end{subfigure}
\hfill\null
\caption{The two relevant cosheaves for computing $H_1(\mathscr{R}^{\epsilon}).$}\label{fig:Cosheaves}
\end{figure}

One can verify that Equation (\ref{DistributedComputationK}) holds by computing
\begin{equation}
\label{CosheafHomologies1}
H_0(C_{\bullet}\cosheaf{F}_1^{\epsilon})=0, \quad H_1(C_{\bullet}\cosheaf{F}_0^{\epsilon}) =\field{K}.
\end{equation}
\end{example}

\begin{example}
\label{LargeE}

Consider now a parameter $\epsilon'$ that is larger than the parameter from Example \ref{SmallE}. The Rips complex $\mathscr{R}^{\epsilon'}$ and the Rips system are illustrated in Figure \ref{fig:RipsComplexEx2} and \ref{fig:RipsSystemEx2}.

\begin{figure}[h!]
\begin{subfigure}{0.3\textwidth}
\vfill
\includegraphics[scale=0.3]{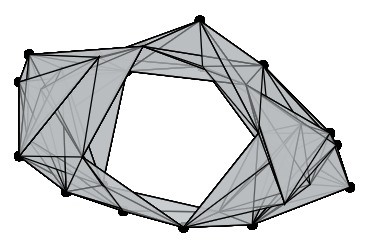}
\vfill
\caption{Rips complex $\mathscr{R}^{\epsilon'}$.}\label{fig:RipsComplexEx2}
\end{subfigure}\hspace{18mm}%
\begin{subfigure}{0.55\textwidth}
\includegraphics[scale=0.2]{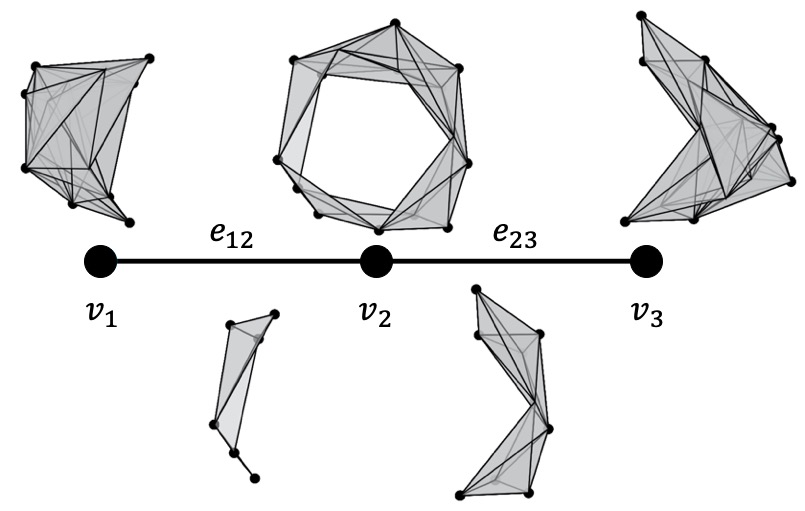}
\caption{Rips system over the nerve $N_{\mathscr{V}}$.}\label{fig:RipsSystemEx2}
\end{subfigure}
\caption{Rips complex and the associated Rips system.}\label{fig:RipsSystem140}
\end{figure}

One can compute cosheaves $\cosheaf{F}_1^{\epsilon'}$, $\cosheaf{F}_0^{\epsilon'}$, and compute the following cosheaf homologies

\begin{equation}
\label{CosheafHomologies2}
H_0(C_{\bullet}\cosheaf{F}_1^{\epsilon'}) = \field{K}, \quad H_1(C_{\bullet}\cosheaf{F}_0^{\epsilon'})=0.
\end{equation}

To compare the cosheaf homologies from Equations (\ref{CosheafHomologies1}) and (\ref{CosheafHomologies2}) for the two parameters $\epsilon < \epsilon'$, note that both $H_1(\mathscr{R}^{\epsilon}) = \field{K}$ and $H_1(\mathscr{R}^{\epsilon'}) = \field{K}$. In Example \ref{SmallE}, the homology class of $H_1(\mathscr{R}^{\epsilon})$ is detected by $H_1(C_{\bullet}\cosheaf{F}_0^{\epsilon})$, while in Example \ref{LargeE}, the homology class of $H_1(\mathscr{R}^{\epsilon'})$ is detected by $H_0(C_{\bullet}\cosheaf{F}_1^{\epsilon'})$. The difference can be explained by comparing the Rips systems from Figures \ref{fig:RipsSystemEx1} and \ref{fig:RipsSystemEx2}. In Figure \ref{fig:RipsSystemEx2}, the Rips complex $\mathscr{R}^{\epsilon'}_{v_2}$ contains a non-trivial 1-cycle, while in Figure \ref{fig:RipsSystemEx1}, there is no such 1-cycle contained in any of the complexes $\mathscr{R}^{\epsilon}_{\sigma}$ for $\sigma \in N_{\mathscr{V}}$.

In general, $H_0(C_{\bullet}\cosheaf{F}_n^{\epsilon})$ reads $n$-cycles that exist in $\mathscr{R}^{\epsilon}_{\sigma}$ for some $\sigma \in N_{\mathscr{V}}$. On the other hand, $H_1(C_{\bullet}\cosheaf{F}_{n-1}^{\epsilon})$ reads $n$-cycles of $\mathscr{R}^{\epsilon}$ that are not cycles of $H_n(\mathscr{R}^{\epsilon}_{\sigma})$ for any $\sigma \in N_{\mathscr{V}}$.
\end{example}

\section{Distributed computation of persistent homology}
\label{sec:DistributedPH}
We restate the main question using the terminology introduced so far.

Given a point cloud $P$, one can build Rips complexes for increasing parameter values $(\epsilon_i)_{i=1}^N$, resulting in the following sequence of Rips complexes and inclusion maps
\[ \mathscr{R}^1 \xhookrightarrow{\iota^1} \mathscr{R}^2 \xhookrightarrow{\iota^2} \cdots \xhookrightarrow{\iota^{N-1}} \mathscr{R}^N.\]
By applying the homology functor of dimension $n$, one obtains the persistence module
\begin{equation}
\label{InterestPM}
\mathbb{V}:  H_{n}(\mathscr{R}^1) \xrightarrow{\iota^1_*} H_{n}(\mathscr{R}^2) \xrightarrow{\iota^2_*} \cdots \xrightarrow{\iota^{N-1}_*} H_{n}(\mathscr{R}^N).
\end{equation}
Assuming a covering $\mathscr{V}$ of the point cloud $P$ that has 1-d nerve (cf. Lemma \ref{DiscreteCGNBound}), we have the following isomorphisms
\begin{equation}
H_n(\mathscr{R}^i) \cong H_0( C_{\bullet} \cosheaf{F}^i_n) \oplus H_1(C_{\bullet}\cosheaf{F}^i_{n-1}) \end{equation}
for every $i$ and $n$. Our main question is stated as the following.

\smallskip
\textit{
Can we construct a persistence module
\[
	\mathbb{V}_{\Psi} : H_0( C_{\bullet} \cosheaf{F}^1_n) \oplus H_1(C_{\bullet}\cosheaf{F}^1_{n-1}) \xrightarrow{\Psi^1} \cdots \xrightarrow{\Psi^{N-1}} H_0( C_{\bullet} \cosheaf{F}^{N}_n) \oplus H_1(C_{\bullet}\cosheaf{F}^{N}_{n-1})
\]
isomorphic to the persistence module $\mathbb{V}$ from Equation (\ref{InterestPM})?
}
\smallskip

In \S\ref{CosheafMorphismSection}, we show that the most naturally induced morphisms of cosheaf homologies are not enough to construct the desired persistence module $\mathbb{V}_{\Psi}$. In \S\ref{SpectralSequenceSection}, we use spectral sequences to construct a map $\psi^i : H_1(C_{\bullet} \cosheaf{F}^i_{n-1}) \to H_0(C_{\bullet} \cosheaf{F}^{i+1}_n)$ that can be used to define the persistence module $\mathbb{V}_{\Psi}$. As it will be illustrated in \S\ref{SpectralSequenceSection}, there are multiple choices involved in defining the map $\psi^i$, and the construction of $\mathbb{V}_{\Psi}$ requires that the maps $\psi^i$ be defined consistently across parameters $(\epsilon_i)_{i=1}^N$. In \S\ref{Construction}, we provide an algorithm to construct the maps $\psi^i$ consistently across parameters $(\epsilon_i)_{i=1}^N$. Furthermore, we construct the persistence module
\[
	\mathbb{V}_{\Psi} : H_0( C_{\bullet} \cosheaf{F}^1_n) \oplus H_1(C_{\bullet}\cosheaf{F}^1_{n-1}) \xrightarrow{\Psi^1} \cdots \xrightarrow{\Psi^{N-1}} H_0( C_{\bullet} \cosheaf{F}^{N}_n) \oplus H_1(C_{\bullet}\cosheaf{F}^{N}_{n-1}).
\]
In \S\ref{IsoPersistenceModules}, we show that the persistence module $\mathbb{V}_{\Psi}$ is isomorphic to the persistence module $\mathbb{V}$ from Equation (\ref{InterestPM}).

\subsection{Cosheaf morphisms}
\label{CosheafMorphismSection}

Given a pair of parameters $\epsilon_i < \epsilon_{i+1}$ and $\sigma \in N_{\mathscr{V}}$, there exist maps $\cosheaf{F}^i_n(\sigma) \to \cosheaf{F}^{i+1}_n(\sigma)$ induced by the inclusion $\mathscr{R}^i_{\sigma} \hookrightarrow \mathscr{R}^{i+1}_{\sigma}$. The collection of such maps is the cosheaf morphism $\phi^i_n : \cosheaf{F}^i_n \to \cosheaf{F}^{i+1}_n$. In particular, the cosheaf morphisms $\phi^i_n$ and $\phi^i_{n-1}$ induce the following morphisms on homology
\begin{eqnarray}
\label{InducedMaps}
\begin{split}
H_0(\phi^i_n) & : H_0(C_{\bullet}\cosheaf{F}^i_n) \to H_0(C_{\bullet}\cosheaf{F}^{i+1}_n),\\
H_1(\phi^i_{n-1}) & : H_1(C_{\bullet}\cosheaf{F}^i_{n-1}) \to H_1(C_{\bullet}\cosheaf{F}^{i+1}_{n-1}).
\end{split}
\end{eqnarray}
The maps $H_0(\phi^i_n)$ and $H_1(\phi^i_{n-1})$ are insufficient to construct a persistence module isomorphic to $\mathbb{V}$ from Equation (\ref{InterestPM}), as we now demonstrate. Using the maps $H_0(\phi^i_n)$ and $H_1(\phi^i_{n-1})$, one defines
\[
	\wmap^i : H_0(C_{\bullet}\cosheaf{F}^i_n) \oplus H_1(C_{\bullet}\cosheaf{F}^i_{n-1}) \to H_0(C_{\bullet}\cosheaf{F}^{i+1}_n) \oplus H_1(C_{\bullet}\cosheaf{F}^{i+1}_{n-1})
\]
by $\wmap^i ( u,v) = ( H_0(\phi^i_n)(u), H_1(\phi^i_{n-1})(v))$.
The obvious attempt to reconstruct $\mathbb{V}$ is the persistence module
\[
	\wmodule : H_0( C_{\bullet} \cosheaf{F}^1_n) \oplus H_1(C_{\bullet}\cosheaf{F}^1_{n-1}) \xrightarrow{\wmap^1} \cdots \xrightarrow{\wmap^{N-1}} H_0( C_{\bullet} \cosheaf{F}^{N}_n) \oplus H_1(C_{\bullet}\cosheaf{F}^{N}_{n-1}).
\]

{\bf Claim:} $\wmodule$ cannot be isomorphic to $\mathbb{V}$ from Equation (\ref{InterestPM}). 

{\em Proof:} The putative isomorphisms $\Phi^i: H_0(C_{\bullet} \cosheaf{F}^i_n) \oplus H_1(C_{\bullet} \cosheaf{F}^i_{n-1}) \to H_n(\mathscr{R}^i)$ would yield commutative diagrams
\begin{equation}
\label{CommutativeDiagram1}
\begin{tikzcd}
H_0(C_{\bullet} \cosheaf{F}^i_n) \oplus H_1(C_{\bullet} \cosheaf{F}^i_{n-1}) \arrow{d}{\wmap^i} \arrow{r}{\Phi^i} & H_n(\mathscr{R}^i) \arrow{d}{\iota^i_*} \\
H_0(C_{\bullet} \cosheaf{F}^{i+1}_n) \oplus H_1(C_{\bullet} \cosheaf{F}^{i+1}_{n-1}) \arrow{r}{\Phi^{i+1}} & H_n(\mathscr{R}^{i+1})
\end{tikzcd}
\end{equation}

However, Examples \ref{SmallE} and \ref{LargeE} illustrate a situation where it is impossible to find isomorphisms $\Phi^i$ and $\Phi^{i+1}$ that make Diagram \ref{CommutativeDiagram1} commute.
Assume that there exists an isomorphism $\Phi^i$, and let $(0,s) \in H_0(C_{\bullet} \cosheaf{F}^i_1) \oplus H_1(C_{\bullet} \cosheaf{F}^i_0)$ be the element such that $\Phi^i(0,s)$ represents the non-trivial $1$-cycle in Figure \ref{fig:RipsComplexEx1}. Then, $\iota^i_* \circ \Phi^i (0,s)$ must be the non-trivial 1-cycle in Figure \ref{fig:RipsComplexEx2}. On the other hand, with our current construction of $\wmap^i$, we have $\wmap^i(0,s)=0$, and hence, $\Phi^{i+1} \circ \wmap^i (0,s) =0$ for any isomorphism $\Phi^{i+1}$. Thus, there are no isomorphisms $\Phi^i$ and $\Phi^{i+1}$ that make Diagram \ref{CommutativeDiagram1} commute.

The core reason why
Diagram \ref{CommutativeDiagram1} fails to commute is that
as we increase the parameter from $\epsilon_i$ to $\epsilon_{i+1}$, a cycle in $H_1(C_{\bullet}\cosheaf{F}^i_{n-1})$ can become homologous to a cycle in $H_0(C_{\bullet}\cosheaf{F}^{i+1}_n)$. The current construction of $\wmap^i$ fails to take such subtlety into account. This motivates our technique: we construct a map from $H_1(C_{\bullet}\cosheaf{F}^i_{n-1})$ to $H_0(C_{\bullet}\cosheaf{F}^{i+1}_n)$.

\subsection{Connecting morphism via spectral sequences}
\label{SpectralSequenceSection}

We seek a map $\psi^i : H_1(C_{\bullet} \cosheaf{F}^i_{n-1}) \to H_0(C_{\bullet} \cosheaf{F}^{i+1}_n)$  for the construction of the persistence module $\mathbb{V}_{\Psi}$.
The plan to build $\psi^i$ is as an extension of a map $\delta^i: \ker H_1(\phi^i_{n-1}) \to H_0(C_{\bullet} \cosheaf{F}^{i+1}_n)$ using a spectral sequence type argument.


\begin{Theorem}
\label{ConnectingHomThm2}
Let $P$ be a point cloud with finite cover $\mathscr{V}$ having one-dimensional nerve $N_{\mathscr{V}}$. There exists a morphism $\delta^i : \ker H_1( \phi^i_{n-1}) \to H_0(C_{\bullet} \cosheaf{F}^{i+1}_n)$ induced by cosheaf morphisms $\phi^i_n$ and $\phi^i_{n-1}$.
\end{Theorem}

\begin{proof}
\sloppy Consider the following commutative diagram. Let $\incv^i_n :\bigoplus\limits_{v \in N_{\mathscr{V}}} C_n(\mathscr{R}^i_v) \to \bigoplus\limits_{v \in N_{\mathscr{V}}} C_n(\mathscr{R}^{i+1}_v)$ denote the collection of inclusion maps of the Rips complexes over the vertices of $N_{\mathscr{V}}$, and let $\ince^i_n : \bigoplus\limits_{e \in N_{\mathscr{V}}} C_n(\mathscr{R}^i_e) \to \bigoplus\limits_{e \in N_{\mathscr{V}}} C_n(\mathscr{R}^{i+1}_e)$ denote the collection of inclusion maps of the Rips complexes over the edges of $N_{\mathscr{V}}$. Let $e^i_n:\bigoplus\limits_{e \in N_{\mathscr{V}}}C_n(\mathscr{R}^i_e) \to \bigoplus\limits_{v \in N_{\mathscr{V}}} C_n(\mathscr{R}^i_v)$ denote the collection of inclusion maps.
The front and back faces of the cube in Diagram \ref{SS0} are the $0^{\textrm{th}}$ pages of the spectral sequence in the proof of Lemma \ref{DiscreteCGNBound} for parameters $\epsilon_i$ and $\epsilon_{i+1}$ respectively.

\begin{equation}
\label{SS0}
\begin{tikzcd}[column sep=small]
 \quad &\quad \arrow{d} & \quad & \quad \arrow{d} \\
\quad \arrow{d} & \bigoplus\limits_{v \in N_{\mathscr{V}}} C_n(\mathscr{R}^{i+1}_v)   \arrow[dd,"\partial", near start] \arrow[from=rr,"e^{i+1}_n" near end] & \quad \arrow{d} &\bigoplus\limits_{e \in N_{\mathscr{V}}} C_n(\mathscr{R}^{i+1}_{e})  \arrow[dd,"\partial", near start]  \\
 \bigoplus\limits_{v \in N_{\mathscr{V}}} C_{n}(\mathscr{R}^i_v) \arrow{ur}{\incv^i_n} \arrow[dd, "\partial", near start ] \arrow[from=rr, crossing over,"e^i_n" near end] & &  \bigoplus\limits_{e \in N_{\mathscr{V}}} C_{n}(\mathscr{R}^i_{e})  \arrow{ur}{\ince^i_{n}}  \\
 &   \bigoplus\limits_{v \in N_{\mathscr{V}}} C_{n-1}(\mathscr{R}^{i+1}_v) \arrow{d}   \arrow[from=rr, "e^{i+1}_{n-1}" near end] & & \bigoplus\limits_{e \in N_{\mathscr{V}}} C_{n-1}(\mathscr{R}^{i+1}_{e}) \arrow{d}  \\
 \bigoplus\limits_{v \in N_{\mathscr{V}}} C_{n-1}(\mathscr{R}^i_v) \arrow{d} \arrow[from=rr,"e^i_{n-1}" near end]  \arrow{ur}{\incv^i_{n-1}} & \quad &    \bigoplus\limits_{e \in N_{\mathscr{V}}} C_{n-1}(\mathscr{R}^i_{e}) \arrow{ur}{\ince^i_{n-1}} \arrow{d} \arrow[from=uu, crossing over,"\partial", near start] & \quad \\
 \quad &\quad & \quad & \quad
\end{tikzcd}
\end{equation}

Computing the homology with respect to the boundary maps $\partial$ yields Diagram \ref{SS1}, in which the maps $\partial_n$ are boundary maps of the chain complexes $C_{\bullet} \cosheaf{F}^i_n$ of the respective cosheaves.

\begin{equation}
\label{SS1}
\begin{tikzcd}[column sep=small]
 \quad &\quad \dar[dashed,dash] & \quad & \quad \dar[dashed,dash] \\
\quad \arrow[d, dashed, dash] & \bigoplus\limits_{v \in N_{\mathscr{V}}} H_n(\mathscr{R}^{i+1}_v)   \arrow[dd,dashed,dash] \arrow[from=rr,"\partial^{i+1}_n" near end] & \quad \arrow[d,dashed,dash] &\bigoplus\limits_{e \in N_{\mathscr{V}}} H_n(\mathscr{R}^{i+1}_{e})  \arrow[dd,dashed,dash]  \\
 \bigoplus\limits_{v \in N_{\mathscr{V}}} H_{n}(\mathscr{R}^i_v) \arrow{ur}{(\phi^{i}_n)_v} \arrow[dd, dashed,dash ] \arrow[from=rr, crossing over,"\partial^i_n" near end] & &  \bigoplus\limits_{e \in N_{\mathscr{V}}} H_{n}(\mathscr{R}^i_{e})  \arrow{ur}{(\phi^{i}_{n})_e}  \\
 &   \bigoplus\limits_{v \in N_{\mathscr{V}}} H_{n-1}(\mathscr{R}^{i+1}_v) \arrow[d,dashed,dash]   \arrow[from=rr, "\partial^{i+1}_{n-1}" near end] & & \bigoplus\limits_{e \in N_{\mathscr{V}}} H_{n-1}(\mathscr{R}^{i+1}_{e}) \arrow[d,dashed,dash]  \\
 \bigoplus\limits_{v \in N_{\mathscr{V}}} H_{n-1}(\mathscr{R}^i_v) \dar[dashed,dash] \arrow[from=rr,"\partial^i_{n-1}" near end]  \arrow{ur}{(\phi^{i}_{n-1})_v} & \quad &    \bigoplus\limits_{e \in N_{\mathscr{V}}} H_{n-1}(\mathscr{R}^i_{e}) \arrow{ur}{(\phi^{i}_{n-1})_e} \dar[dashed,dash] \arrow[from=uu, crossing over, dashed, dash] & \quad \\
 \quad &\quad & \quad & \quad
\end{tikzcd}
\end{equation}

Computing the homology with respect to these maps $\partial^i_n$ yields Diagram \ref{SS2} of cosheaf homologies. 

\begin{equation}
\label{SS2}
\begin{tikzcd}
 \quad &\quad \dar[dashed,dash] & \quad & \quad \dar[dashed,dash] \\
\quad \arrow[d, dashed, dash] & H_0(C_{\bullet} \cosheaf{F}^{i+1}_n)  \arrow[dd,dashed,dash] \arrow[from=rr,dashed,dash] & \quad \arrow[d,dashed,dash] &H_1(C_{\bullet} \cosheaf{F}^{i+1}_n)  \arrow[dd,dashed,dash]  \\
 H_0(C_{\bullet} \cosheaf{F}^i_n) \arrow{ur}{H_0(\phi^i_n)} \arrow[dd, dashed,dash ] \arrow[from=rr, crossing over,dashed,dash] & & H_1(C_{\bullet} \cosheaf{F}^i_n) \arrow{ur}{H_1(\phi^i_{n})}  \\
 &   H_0(C_{\bullet} \cosheaf{F}^{i+1}_{n-1})  \arrow[d,dashed,dash]   \arrow[from=rr, dashed,dash] & & H_1(C_{\bullet} \cosheaf{F}^{i+1}_{n-1})  \arrow[d,dashed,dash]  \\
 H_0(C_{\bullet} \cosheaf{F}^i_{n-1}) \dar[dashed,dash] \arrow[from=rr,dashed,dash]  \arrow{ur}{H_0(\phi^i_{n-1})} & \quad &    H_1(C_{\bullet} \cosheaf{F}^i_{n-1})   \arrow{ur}{H_1(\phi^i_{n-1})} \dar[dashed,dash] \arrow[from=uu, crossing over, dashed, dash] & \quad \\
 \quad &\quad & \quad & \quad
\end{tikzcd}
\end{equation}

To continue, some notation is necessary to distinguish where homology classes reside. Let $\langle$ $\rangle$ and $\{$ $\}$ denote the homology classes that appear in Diagrams \ref{SS1} and $\ref{SS2}$ respectively. For example, if $\gamma \in \bigoplus\limits_{e \in N_{\mathscr{V}}} C_{n-1}(\mathscr{R}^i_e)$ and $\partial \gamma=0$, then $\langle \gamma \rangle$ denotes the homology class of $\gamma$ in $\bigoplus\limits_{e \in N_{\mathscr{V}}} H_{n-1}(\mathscr{R}^i_e)$. Furthermore, if $\partial^i_{n-1} \langle \gamma \rangle=0$, then $\{ \langle \gamma \rangle \}$ denotes the homology class of $\langle \gamma \rangle$ in $H_1(C_{\bullet} \cosheaf{F}^i_{n-1})$.

With this notation in place, we define a map $\delta^i : \ker H_1(\phi^i_{n-1}) \to H_0(C_{\bullet} \cosheaf{F}^{i+1}_n)$ on a basis $\mathscr{B}^i_{\ker}$ of $\ker H_1(\phi^i_{n-1})$. For each basis element $\{ \langle b \rangle \} \in \mathscr{B}^i_{\ker}$, fix a coset representative $b^*$ of $\langle b \rangle$. Since $\{ \langle b \rangle \} \in \ker H_1(\phi^i_{n-1})$, we know that
$(\phi^{i}_{n-1})_e \langle b \rangle$ is trivial in $\bigoplus\limits_{e \in N_{\mathscr{V}}} H_{n-1}(\mathscr{R}^{i+1}_e)$. Thus, there exists $\alpha^{i+1} \in \bigoplus\limits_{e \in N_{\mathscr{V}}} C_n(\mathscr{R}^{i+1}_e)$ such that
\begin{equation}
\label{Alpha}
\partial \alpha^{i+1}=\ince^i_{n-1} b^* .
\end{equation}
Moreover,
since $\partial^i_{n-1} \langle b \rangle =0$,
there exists $\beta^i \in \bigoplus\limits_{v \in N_{\mathscr{V}}} C_n(\mathscr{R}^i_{v})$ such that
$\partial \beta^i=e_{n-1}^i b^*$.

With this, we now define
\begin{equation}
\label{deltai_def_basis}
\delta^i \{ \langle b \rangle \} = \{ \langle -e^{i+1}_n \circ \alpha^{i+1} +  \incv^i_n \circ \beta^i \rangle \}.
\end{equation}
One can check that $-e^{i+1}_n \circ \alpha^{i+1} +  \incv^i_n \circ \beta^i$ represents an element in $H_0(C_{\bullet} \cosheaf{F}^{i+1}_n)$.
%
Extending linearly from the basis gives a morphism
$\delta^i : \ker H_1(\phi^i_{n-1}) \to H_0(C_{\bullet} \cosheaf{F}^{i+1}_n)$.

Note that the construction of $\delta^i$ involved a choice of basis $\mathscr{B}^i_{\ker}$ of $\ker H_1(\phi^i_{n-1})$, its coset representatives, and a choice of $\alpha^{i+1}$ and $\beta^i$ for each basis vector $\{ \langle b \rangle \}$ of $\ker H_1(\phi^i_{n-1}) $. One can check that different choices of $\alpha^{i+1}$ do not affect the map $\delta^i$ (Appendix \ref{DeltaWellDefined}). However, the different choices of $\beta^i$ does affect the map $\delta^i$. In \S\ref{Construction}, we construct the map $\delta^i$ by carefully choosing the basis $\mathscr{B}^i_{\ker}$, its coset representatives, and $\beta^i$.
\end{proof}

Once we define the map $\delta^i$, we can extend the map to $\psi^i:H_1(C_{\bullet} \cosheaf{F}^i_{n-1}) \to H_0(C_{\bullet} \cosheaf{F}^{i+1}_n)$ as the following. Note that
\begin{equation}
\label{H1Decomposition}
H_1 (C_{\bullet} \cosheaf{F}^i_{n-1}) = A^i \oplus \ker H_1(\phi^i_{n-1})
\end{equation}
for some subspace $A^i$. Then every $\{ \langle y \rangle \} \in H_1(C_{\bullet} \cosheaf{F}^i_{n-1})$ can be written uniquely as $\{ \langle y \rangle \}= \{ \langle y_1 \rangle \} + \{ \langle y_2 \rangle \}$, with $ \{ \langle y_1 \rangle \} \in A^i$ and $\{ \langle y_2 \rangle \} \in \ker H_1(\phi^i_{n-1})$. Extend the map $\delta^i$ to $\psi^i : H_1(C_{\bullet} \cosheaf{F}^i_{n-1}) \to H_0(C_{\bullet} \cosheaf{F}^{i+1}_n)$ by
\begin{equation}
\label{PsiExtension}
\psi^i \{ \langle y \rangle \}  = \psi^i(\{ \langle y_1 \rangle \} + \{ \langle y_2 \rangle \} ) = \delta^i\{ \langle y_2 \rangle \}.
\end{equation}

\subsection{Construction of distributed persistence module}
\label{Construction}

In this section, we provide an algorithmic way of making consistent choices across parameters $(\epsilon_i)_{i=1}^N$ so that we can define the maps $\delta^i$ and $\psi^i$ consistently across parameters. The resulting collection of maps $(\psi^i)_{i=1}^N$ will then be used to construct the desired persistence module $\mathbb{V}_{\Psi}$.

We will inductively fix a basis $\mathscr{B}^i_{\ker}$ of $\ker H_1(\phi^i_{n-1})$ and extend it to a basis $\mathscr{B}^i$ of $H_1(C_{\bullet} \cosheaf{F}^i_{n-1})$. We will define a set map $\Gamma^i : \mathscr{B}^i \to \bigoplus\limits_{v \in N_{\mathscr{V}}} C_n(\mathscr{R}^i_v)$ that consistently chooses $\beta^i$'s  for each element of $\mathscr{B}^i$. We will then define $\delta^i$ on the basis $\{ \langle b \rangle \} \in \mathscr{B}^i_{\ker}$ by
\begin{equation}
\label{RealDelta}
\delta^i \{ \langle b  \rangle \} =\{ \langle \, -e^{i+1}_n \alpha^{i+1} + \incv^i_n \circ \, \Gamma^i \{ \langle b \rangle \} \, \rangle \}
\end{equation}
and extend the map linearly to $\ker H_1(\phi^i_{n-1})$. The map $\delta^i$ can then be extended to $\psi^i:H_1(C_{\bullet} \cosheaf{F}^i_{n-1}) \to H_0(C_{\bullet} \cosheaf{F}^{i+1}_n)$ as Equation (\ref{PsiExtension}).

Note that the construction of $\delta^i: \ker H_1(\phi^i_{n-1}) \to H_0(C_{\bullet} \cosheaf{F}^{i+1}_n)$ requires a choice of $\beta^i$ for only the basis elements $\{ \langle b \rangle \} \in \mathscr{B}^i_{\ker}$. However, we choose such $\beta^i$ for every basis element $\{ \langle b \rangle \} \in \mathscr{B}^i$ because such choice can affect the construction of $\delta^j$ for $j > i$.

\subsubsection*{Base case}

Recalling Diagram \ref{SS0},
%
%
%
note that
\[ H_1(C_{\bullet} \cosheaf{F}^1_{n-1}) = A^1 \oplus \ker H_1(\phi^1_{n-1})\]
for some subspace $A^1$. Let $\mathscr{B}^1_A$ be a basis of $A^1$, and let $\mathscr{B}^1_{\ker}$ be a basis of $\ker H_1(\phi^1_{n-1})$. Then,
\begin{equation}
\label{BaseCaseB}
\mathscr{B}^1 = \mathscr{B}^1_A \cup \mathscr{B}^1_{\ker}
\end{equation}
is a basis of $H_1(C_{\bullet} \cosheaf{F}^1_{n-1})$. For each basis element $\{ \langle b \rangle \} \in \mathscr{B}^1$, fix a coset representative $b^*$ of $\langle b \rangle$.

Define a set map $\Gamma^1: \mathscr{B}^1 \to \bigoplus\limits_{v \in N_{\mathscr{V}}} C_n(\mathscr{R}^1_v)$ as following. For each $\{ \langle b \rangle \} \in \mathscr{B}^1$, let
\begin{equation}
\label{Correspondence1}
\Gamma^1 \{ \langle b \rangle \} = \beta^1,
\end{equation}
where $\beta^1 \in \bigoplus\limits_{v \in N_{\mathscr{V}}} C_n(\mathscr{R}^1_v)$ is any element satisfying $\partial \beta^1 = e^1_{n-1} b^*$. 

Define $\delta^1$ on $\{ \langle b \rangle \} \in \mathscr{B}^1_{\ker}$ by
\[ \delta^1 \{ \langle b  \rangle \} =\{ \langle \, -e^{2}_n \alpha^{2} + \incv^1_n \circ \, \Gamma^1 \{ \langle b \rangle \} \, \rangle \}, \]
where $\alpha^2 \in \bigoplus\limits_{e \in N_{\mathscr{V}}} C_n(\mathscr{R}^{i+1}_e)$
is any element satisfying Equation (\ref{Alpha}).
Extend $\delta^1$ linearly to $\ker H_1(\phi^1_{n-1})$. The map $\delta^1$ can be used to define the map $\psi^1 : H_1(C_{\bullet} \cosheaf{F}^1_{n-1}) \to H_0(C_{\bullet} \cosheaf{F}^2_n)$ via Equation (\ref{PsiExtension}).

\subsubsection*{Inductive step}

\begin{itemize}
\renewcommand{\labelitemi}{$\bullet$}
\item \textbf{Inductive assumption.}

Note that
\[ H_1(C_{\bullet} \cosheaf{F}^{i-1}_{n-1}) = A^{i-1} \oplus \ker H_1(\phi^{i-1}_{n-1})\]
for some subspace $A^{i-1}$.
\begin{itemize}
\item Assume that there exists a basis $\mathscr{B}^{i-1}$ of $H_1(C_{\bullet} \cosheaf{F}^{i-1}_{n-1})$ that has the form
\[\mathscr{B}^{i-1} = \mathscr{B}^{i-1}_A \cup \mathscr{B}^{i-1}_{\ker}, \]
where $\mathscr{B}^{i-1}_A$ is a basis of $A^{i-1}$ and $\mathscr{B}^{i-1}_{\ker}$ is a basis of $\ker H_1(\phi^{i-1}_{n-1})$.
\item Assume that for each basis element $\{ \langle b \rangle \} \in \mathscr{B}^{i-1}$, a coset representative $b^*$ of $\langle b \rangle$ has been fixed.
\item Assume that there exists a set map $\Gamma^{i-1} : \mathscr{B}^{i-1} \to \bigoplus\limits_{v \in N_{\mathscr{V}}} C_n(\mathscr{R}^{i-1}_v)$ such that
\begin{equation}
\label{InductionAssumption}
\partial \Gamma^{i-1} \{ \langle b \rangle \} = e^{i-1}_{n-1} b^*
\end{equation}
for every $ \{ \langle b \rangle \} \in \mathscr{B}^{i-1}$.
\end{itemize}

\item \textbf{Step 1. Fix a basis $\mathscr{C}^i$ of $H_1(C_{\bullet} \cosheaf{F}^i_{n-1})$ that is compatible with $\mathscr{B}^{i-1}$.}

\sloppy By assumption, the basis $\mathscr{B}^{i-1}$ of $H_1(C_{\bullet} \cosheaf{F}^{i-1}_{n-1})$ has the form $\mathscr{B}^{i-1} = \mathscr{B}^{i-1}_A \cup \mathscr{B}^{i-1}_{\ker}$. Without loss of generality, assume that
\[ \mathscr{B}^{i-1}_A = \{ \, \{ \langle b_1 \rangle \}, \dots, \{ \langle b_t \rangle \} \, \}. \]
One can show that $\{ \langle \ince^{i-1}_{n-1} b_1 \rangle \}, \dots, \{ \langle \ince^{i-1}_{n-1} b_t \rangle \}$ are linearly independent in $H_1(C_{\bullet} \cosheaf{F}^i_{n-1})$ (Appendix \ref{LinearIndependence}). Let
\[ \mathscr{C}^i_{\im} =  \{ \, \{ \langle \ince^{i-1}_{n-1} b_1 \rangle \}, \dots, \{ \langle \ince^{i-1}_{n-1} b_t \rangle \} \, \}.\]
Extend $\mathscr{C}^i_{\im}$ to a basis $\mathscr{C}^i$ of $H_1(C_{\bullet} \cosheaf{F}^i_{n-1})$. Let $\mathscr{C}^i_D$ denote the basis vectors of $\mathscr{C}^i$ that are not in $\mathscr{C}^i_{\im}$, i.e.,
\begin{equation}
\label{BasisC}
\mathscr{C}^i = \mathscr{C}^i_{\im} \cup \mathscr{C}^i_D.
\end{equation}

If $\{ \langle c \rangle \} \in \mathscr{C}^i_{\im}$ such that $\{ \langle c \rangle \} = \{ \langle \ince^{i-1}_{n-1} b \rangle \}$, then let $ c^*= \ince^{i-1}_{n-1} b^*$ be the coset representative of $\langle c \rangle$. If $\{ \langle c \rangle \} \in \mathscr{C}^i_D$, fix any coset representative $c^*$ of $\langle c \rangle$.

\item \textbf{Step 2. Define a set map $\Gamma^i_{\mathscr{C}} : \mathscr{C}^i \to \bigoplus\limits_{v \in N_{\mathscr{V}}} C_n(\mathscr{R}^i_v)$.}

We will define a set map $\Gamma^i_{\mathscr{C}}: \mathscr{C}^i \to \bigoplus\limits_{v \in N_{\mathscr{V}}} C_n(\mathscr{R}^i_v)$ such that
\begin{equation}
\label{GammaOnCReq}
\partial \Gamma^i_{\mathscr{C}} \{ \langle c \rangle \} = e^i_{n-1} c^*
\end{equation}
for every $\{ \langle c \rangle \} \in \mathscr{C}^i$. Define
\begin{equation}
\label{GammaOnC}
\Gamma^i_{\mathscr{C}} \{ \langle c \rangle \} =
\begin{cases}
\incv^{i-1}_n \Gamma^{i-1} \{ \langle b \rangle \} & \text{if } \{ \langle c \rangle \} = \{ \langle \ince^{i-1}_{n-1} b \rangle \} \in \mathscr{C}^i_{\im} \\
\text{any } \beta^i \text{ satisfying } \partial \beta^i = e^i_{n-1} c^* & \text{if } \{ \langle c \rangle \} \in \mathscr{C}^i_D
\end{cases}
\end{equation}

\item \textbf{Step 3. Fix a new basis $\mathscr{B}^i$ of $H_1(C_{\bullet} \cosheaf{F}^i_{n-1})$}

Note that
\[ H_1(C_{\bullet} \cosheaf{F}^{i}_{n-1}) = A^{i} \oplus \ker H_1(\phi^{i}_{n-1})\]
for some subspace $A^i$. Let $\mathscr{B}^i_A$ be a basis of $A^i$, and let $\mathscr{B}^i_{\ker}$ be a basis of $\ker H_1(\phi^{i}_{n-1})$. Then,
\begin{equation}
\label{BasisB}
\mathscr{B}^i = \mathscr{B}^i_{A} \cup \mathscr{B}^{i}_{\ker}
\end{equation}
is a basis of $H_1(C_{\bullet} \cosheaf{F}^i_{n-1})$.

The coset representative $b^*$ for each basis vector $\{ \langle b \rangle \} \in \mathscr{B}^i$ follows naturally from the coset representatives of $\mathscr{C}^i$. In other words, if $\mathscr{C}^i =\{ \, \{ \langle c_1 \rangle \} , \dots, \{ \langle c_l \rangle \} \, \}$ and $\{ \langle b \rangle \} \in \mathscr{B}^i$ is written as
\[ \{ \langle b \rangle \} =  d_1 \{ \langle c_1 \rangle \} + \dots + d_l \{ \langle c_l \rangle \}, \]
then let
\begin{equation}
\label{BasisBC}
b^* = d_1 c_1^* + \dots + d_l c_l^*
\end{equation}
be the coset representative of $\langle b \rangle$.
\item \textbf{Step 4. Define the set map $\Gamma_i : \mathscr{B}^i \to \bigoplus\limits_{v \in N_{\mathscr{V}}} C_n(\mathscr{R}^i_v)$. }

Given $\{ \langle b \rangle \} \in \mathscr{B}^i$, if $\{ \langle b \rangle \}$ is written as
\[ \{ \langle b \rangle \} = d_1 \{ \langle c_1 \rangle \} + \cdots + d_l \{ \langle c_l \rangle \}\]
where $\{ \langle c_1 \rangle \} \dots, \{ \langle c_l \rangle \}$ are basis $\mathscr{C}^i$, then define $\Gamma^i \{ \langle b \rangle \}$ as
\[ \Gamma^i \{ \langle b \rangle \} = d_1 \Gamma^i_{\mathscr{C}} \{ \langle c_1 \rangle \} + \cdots + d_l \Gamma^i_{\mathscr{C}} \{ \langle c_l \rangle \}. \]
One can check that
\begin{equation}
\label{KeyEq}
\partial \Gamma^i \{ \langle b \rangle \} = e^i_{n-1} b^*.
\end{equation}

\item \textbf{Step 5. Define the maps $\delta^i$ and $\psi^i$.}

For each $\{ \langle b \rangle \} \in \mathscr{B}^i_{\ker}$, let
\[ \delta^i \{ \langle b \rangle \} = \{ \langle -e^{i+1}_n \alpha^{i+1} + \ince^{i}_n \circ \Gamma^i \{ \langle b \rangle \} \, \rangle \} , \]
where $\alpha^{i+1} \in \bigoplus\limits_{e \in N_{\mathscr{V}}} C_n(\mathscr{R}^{i+1}_e)$ is any element satisfying $\partial \alpha^{i+1}=\ince^i_{n-1} b^*.$

Extend $\delta^i$ linearly to $\delta^i : \ker H_1(\phi^i_{n-1}) \to H_0(C_{\bullet} \cosheaf{F}^{i+1}_n)$. One can then extend $\delta^i$ to map $\psi^i : H_1(C_{\bullet} \cosheaf{F}^i_{n-1}) \to H_0(C_{\bullet} \cosheaf{F}^{i+1}_n)$ via Equation (\ref{PsiExtension}).

\end{itemize}

Going through Steps 1-5 defines $\delta^i$ and $\psi^i$ inductively for every $i$. We can then define the map
\[ \Psi^i : H_0(C_{\bullet} \cosheaf{F}^i_n) \oplus H_1(C_{\bullet} \cosheaf{F}^{i}_{n-1}) \to H_0(C_{\bullet} \cosheaf{F}^{i+1}_n) \oplus H_1(C_{\bullet} \cosheaf{F}^{i+1}_{n-1})\]
by
\begin{equation}
\label{PsiDefinition}
\Psi^{i} ( \{ \langle x \rangle \}, \{ \langle y \rangle \}) = (\,H_0(\phi^{i}_n) \{ \langle x \rangle \} + (-1)^{n+1} \psi^{i} \{ \langle y \rangle \} \,,\, H_1(\phi^{i}_{n-1}) \{ \langle y \rangle \} \,),
\end{equation}
where $H_0(\phi^i_n)$ and $H_1(\phi^{i}_{n-1})$ are the maps defined in Equation (\ref{InducedMaps}). We can then define the persistence module $\mathbb{V}_{\Psi}$
\begin{equation}
\label{PerPsi}
\mathbb{V}_{\Psi} : H_0(C_{\bullet} \cosheaf{F}^1_n) \oplus H_1(C_{\bullet} \cosheaf{F}^1_{n-1}) \xrightarrow{\Psi^1}  \cdots \xrightarrow{\Psi^{N-1}} H_0(C_{\bullet} \cosheaf{F}^N_n) \oplus H_1(C_{\bullet} \cosheaf{F}^N_{n-1}).
\end{equation}
\subsection{Isomorphism of persistence modules}
\label{IsoPersistenceModules}
We show that the persistence module $\mathbb{V}_{\Psi}$ constructed in Equation (\ref{PerPsi}) is isomorphic to the persistence module $\mathbb{V}$ from Equation (\ref{InterestPM}). In order to do so, we will show that both $\mathbb{V}_{\Psi}$ and $\mathbb{V}$ are isomorphic to the following persistence module
\[\mathbb{V}_{\total} : H_n(\total^1) \xrightarrow{\iota^1_{\total}} H_n(\total^2) \xrightarrow{\iota^2_{\total}} \cdots \xrightarrow{\iota^{N-1}_{\total}} H_n(\total^{N}), \]
where each $H_n(\total^i)$ is the homology of the double complex from Diagram \ref{Spectral0'} for parameter $\epsilon_i$, and $\iota^i_{\total}$ is the morphism induced by maps of double complexes.
\begin{equation}
\label{Spectral0'}
\begin{tikzcd}
 \arrow{d}{\partial} \vdots &  \arrow{d}{\partial} \vdots  \\
\dar{\partial} \bigoplus\limits_{v \in N_{\mathscr{V}}} C_2(\mathscr{R}^i_v) & \arrow{l}{e^i_2} \dar{\partial} \bigoplus\limits_{e \in N_{\mathscr{V}}} C_2( \mathscr{R}^i_e) \\
 \dar{\partial} \bigoplus\limits_{v \in N_{\mathscr{V}}} C_1( \mathscr{R}^i_v ) & \arrow{l}{e^i_1} \dar{\partial} \bigoplus\limits_{e \in N_{\mathscr{V}}} C_1(\mathscr{R}^i_e)\\
\bigoplus\limits_{v \in N_{\mathscr{V}}} C_0(\mathscr{R}^i_v) & \arrow{l}{e^i_0}  \bigoplus\limits_{e \in N_{\mathscr{V}}} C_0(\mathscr{R}^i_e) \\
\end{tikzcd}
\end{equation}
Let $H_n(\total^i)$ denote the homology of the double complex. Note that a coset of $H_n(\total^i)$ is represented by $[x,y]$, where $x \in \bigoplus\limits_{v \in N_{\mathscr{V}}} C_n(\mathscr{R}^i_v)$, $y \in \bigoplus\limits_{e \in N_{\mathscr{V}}} C_{n-1}(\mathscr{R}^i_e)$, $\partial y =0$ and $\partial x =(-1)^{n-1} e^i_{n-1} y$. A coset $[x,y]$ is trivial in $H_n(\total^i)$ if there exist $p_{n+1} \in \bigoplus\limits_{v \in N_{\mathscr{V}}} C_{n+1}(\mathscr{R}^i_v)$ and $q_n \in \bigoplus\limits_{e \in N_{\mathscr{V}}} C_n(\mathscr{R}^i_e)$ such that $\partial q_n =  y$ and $\partial p_{n+1}+ (-1)^{n+1} e^i_n (q_n) =x$.

Given increasing parameter values $( \epsilon_i)_{i=1}^N$, one can construct a double complex for each parameter $\epsilon_i$. There exists an inclusion map from double complex associated with parameter $\epsilon_i$ to that of parameter $\epsilon_{i+1}$, as illustrated in
Diagram \ref{SS0}. 
The vertical maps $\incv_n^i$ and $\ince_n^i$ constitute the inclusion maps of double complexes. Such inclusion of double complexes induces a morphism $\iota^i_{\total} :H_n(\total^i) \to H_n(\total^{i+1})$ which can be written explicitly as
\begin{equation}
\label{InducedTotalHomology}
\iota_{\total}^i ([x,y]) = [ \incv^i_n(x), \ince^i_{n-1}(y)].
\end{equation}

\begin{Lemma}
\label{Theorem2}
The persistence modules $\mathbb{V}_{\emph{Tot}}$ and $\mathbb{V}$ are isomorphic.
\end{Lemma}

\begin{proof}
We will define isomorphisms $\Phi^i_{\total} : H_n(\total^i) \to H_n(\mathscr{R}^i)$ such that the following diagram commutes.
\begin{equation}
\label{Diagram2}
\begin{tikzcd}
H_n(\total^1)  \arrow{d}{\Phi^1_{\total}} \arrow{r}{\iota^1_{\total}} &  H_n(\total^2) \arrow{d}{\Phi^{2}_{\total}} \arrow{r}{\iota^2_{\total}} & \cdots  \arrow{r}{\iota^{N-1}_{\total}} & H_n(\total^N) \arrow{d}{\Phi^N_{\total}} \\
H_n(\mathscr{R}^1)   \arrow{r}{\iota^1_*} & H_n(\mathscr{R}^{2})  \arrow{r}{\iota^2_*}  & \cdots \arrow{r}{\iota^{N-1}_*} & H_n(\mathscr{R}^{N})
\end{tikzcd}
\end{equation}

For each parameter $\epsilon_i$, let $j^i_n : \bigoplus\limits_{v \in N_{\mathscr{V}}} C_n(\mathscr{R}^i_v) \to C_n(\mathscr{R}^i)$ be a collection of inclusion maps. Define $\Phi^i_{\total}$ by

\begin{equation}
\label{PsiTotal}
\Phi^i_{\total} ([x,y])=[j^i_n(x)].
\end{equation}
One can check that $\Phi^i_{\total}$ is well-defined and bijective (Appendix \ref{ProofTheorem2}).

Given $[x, y] \in H_n(\total^i)$, note that
\[ \iota^i_* \circ \Phi^i_{\total} [x,y]= \iota^i_* [ j^i_n(x) ] = [\iota^i \circ j^i_n(x)],\]
\[ \Phi^{i+1}_{\total} \circ \iota^i_{\total} [x,y]= \Phi^{i+1}_{\total} [\incv^i_n ( x) , \ince^i_{n-1}(y)] = [j^{i+1}_n \circ \incv^i_n(x)].\]
Then, $\iota^i_* \circ \Phi^i_{\total} = \Phi^{i+1}_{\total} \circ \iota^i_{\total}$ because all the maps involved are inclusion maps. Thus, Diagram \ref{Diagram2} commutes.
\end{proof}

We now show that $\mathbb{V}_{\Psi}$ and $\mathbb{V}_{\total}$ are isomorphic persistence modules. Recall that the persistence module $\mathbb{V}_{\Psi}$ is defined as
\begin{equation}
\mathbb{V}_{\Psi} : H_0(C_{\bullet} \cosheaf{F}^1_n) \oplus H_1(C_{\bullet} \cosheaf{F}^1_{n-1}) \xrightarrow{\Psi^1}  \cdots \xrightarrow{\Psi^{N-1}} H_0(C_{\bullet} \cosheaf{F}^N_n) \oplus H_1(C_{\bullet} \cosheaf{F}^N_{n-1}),
\end{equation}
where the maps $\Psi^i$ are defined as
\begin{equation}
\Psi^{i} ( \{ \langle x \rangle \}, \{ \langle y \rangle \}) = (\,H_0(\phi^{i}_n) \{ \langle x \rangle \} + (-1)^{n+1} \psi^{i} \{ \langle y \rangle \} \,,\, H_1(\phi^{i}_{n-1}) \{ \langle y \rangle \} \,),
\end{equation}
where $H_0(\phi^i_n)$ and $H_1(\phi^{i}_{n-1})$ are the maps defined in Equation (\ref{InducedMaps}). Recall from Equation (\ref{H1Decomposition}) that every $\{ \langle y \rangle \} \in H_1(C_{\bullet} \cosheaf{F}^i_{n-1})$ can be expressed explicitly as $\{ \langle y \rangle \} = \{ \langle y_1 \rangle \} + \{ \langle y_2 \rangle \}$, where $\{ \langle y_1 \rangle \} \in A^i$ and $\{ \langle y_2 \rangle \} \in \ker H_1(\phi^i_{n-1})$. Then, we can express the map $\Psi^i$ explicitly using maps from Diagram \ref{SS0} and $\delta^i$ as the following
\begin{equation}
\label{PsiExplicit}
\Psi^{i} ( \{ \langle x \rangle \}, \{ \langle y \rangle \}) = (\, \{ \langle \incv^i_n x \rangle \} + (-1)^{n+1} \delta^i \{ \langle y_2 \rangle\}  \,,\,\{ \langle \ince^i_n y \rangle \} \,)
\end{equation}

\begin{Lemma}
\label{Theorem1}
The persistence modules $\mathbb{V}_{\Psi}$ and $\mathbb{V}_{\emph{Tot}}$ are isomorphic.
\end{Lemma}

\begin{proof}
We will define isomorphisms $\Phi^i : H_0(C_{\bullet} \cosheaf{F}^i_n) \oplus H_1(C_{\bullet} \cosheaf{F}^i_{n-1}) \to H_n(\total^i)$ such that the following diagram commutes.
\begin{equation}
\label{DiagramFirst}
\begin{tikzcd}
H_0(C_{\bullet} \cosheaf{F}^1_n) \oplus H_1(C_{\bullet} \cosheaf{F}^1_{n-1})  \arrow{d}{\Phi^1} \arrow{r}{\Psi^1} &  \cdots  \arrow{r}{\Psi^{N-1}} & H_0(C_{\bullet} \cosheaf{F}^N_n) \oplus H_1(C_{\bullet} \cosheaf{F}^N_{n-1}) \arrow{d}{\Phi^N} \\
H_n(\total^1)   \arrow{r}{\iota^1_{\total}} & \cdots \arrow{r}{\iota^{N-1}_{\total}} & H_n(\total^N)
\end{tikzcd}
\end{equation}

We will define $\Phi^i:H_0(C_{\bullet} \cosheaf{F}^i_n) \oplus H_1(C_{\bullet} \cosheaf{F}^i_{n-1}) \to H_n(\total^i)$ by first defining
\begin{align}
\Phi^i_0: H_0(C_{\bullet} \cosheaf{F}^i_n) \to H_n( \total^i) \label{PhiZero}, \\
\Phi^i_1: H_1(C_{\bullet} \cosheaf{F}^i_{n-1}) \to H_n(\total^i) \label{PhiOne}.
\end{align}

Define $\Phi^i_0: H_0(C_{\bullet} \cosheaf{F}^i_n) \to H_n( \total^i)$ by
\begin{equation}
\label{Phi0}
\Phi^i_0 ( \{ \langle x \rangle \} ) = [ x, 0 ]. 
\end{equation}
To define the map $\Phi^i_1$, we will define $\Phi^i_1$ on the basis $\mathscr{B}^i$ of $H_1(C_{\bullet} \cosheaf{F}^i_{n-1})$. Recall the fixed basis $\mathscr{B}^i$ of $H_1(C_{\bullet} \cosheaf{F}^i_{n-1})$ in Equation (\ref{BasisB}). For each basis $\{ \langle b \rangle \} \in \mathscr{B}^i$, let
\begin{equation}
\label{Phi1}
\Phi^i_1 ( \{ \langle b \rangle \} )= [ (-1)^{n+1} \Gamma^i \{ \langle b \rangle \}, b^* ], 
\end{equation}
where the coset representative $b^*$ and the set map $\Gamma^i$ are defined according to the construction in \S\ref{Construction}. Extend $\Phi^i_1$ linearly to $H_1(C_{\bullet} \cosheaf{F}^i_{n-1})$.

We can now define $\Phi^i:H_0(C_{\bullet} \cosheaf{F}^i_n) \oplus H_1(C_{\bullet} \cosheaf{F}^i_{n-1}) \to H_n(\total^i)$ by
\begin{equation}
\label{Phi01}
\Phi^i( \{ \langle x \rangle \}, \{ \langle y \rangle \}) = \Phi^i_0(\{\langle x \rangle \}) + \Phi^i_1 ( \{ \langle y \rangle \} ).
\end{equation}
One can check that $\Phi^i$ is well-defined and bijective (Appendix \ref{ProofTheorem1}).

To show that Diagram \ref{DiagramFirst} commutes, it suffices to show that the following diagram commutes for each element of $H_0(C_{\bullet} \cosheaf{F}^i_n)$ and $H_1(C_{\bullet} \cosheaf{F}^i_{n-1})$.
\begin{equation}
\label{Commute1}
\begin{tikzcd}
H_0(C_{\bullet} \cosheaf{F}^i_n) \oplus H_1(C_{\bullet} \cosheaf{F}^i_{n-1}) \arrow{d}{\Phi^i} \arrow{r}{\Psi^i} & H_0(C_{\bullet} \cosheaf{F}^{i+1}_n) \oplus H_1(C_{\bullet} \cosheaf{F}^{i+1}_{n-1}) \arrow{d}{\Phi^{i+1}} \\
H_n(\total^i) \arrow{r}{\iota^i_{\total}} & H_n(\total^{i+1})
\end{tikzcd}
\end{equation}

\textbf{Case 1}: Given $\{ \langle x \rangle \} \in H_0(C_{\bullet} \cosheaf{F}^i_n)$, we know from Equations (\ref{InducedTotalHomology}) and (\ref{Phi0}) that
\[ \iota^i_{\total} \circ \Phi^i (\{ \langle x \rangle \},0) = \iota^i_{\total}([x,0]) = [ \incv^i_n x, 0].\]
On the other hand, we know from Equations (\ref{PsiDefinition}) and (\ref{Phi0}) that
\[ \Phi^{i+1} \circ \Psi^i(\{ \langle x \rangle \},0) =\Phi^{i+1}(H_0(\phi^i_n)(x),0)= \Phi^{i+1} ( \{ \langle \incv^i_n x \rangle \},0) =[ \incv^i_n x,0].\]
Thus, the Diagram \ref{Commute1} commutes for every $\{ \langle x \rangle \} \in H_0(C_{\bullet} \cosheaf{F}^i_n)$.

\textbf{Case 2}: To show that Diagram \ref{Commute1} commutes for every vector in $H_1(C_{\bullet} \cosheaf{F}^i_{n-1})$, we will show that Diagram \ref{Commute1} commutes for every basis element $\{ \langle b \rangle \} \in \mathscr{B}^i$ of $H_1(C_{\bullet} \cosheaf{F}^i_{n-1})$. Recall that $\mathscr{B}^i = \mathscr{B}^i_A \cup \mathscr{B}^i_{\ker}$. We consider two cases separately: the first, if $\{ \langle b \rangle \} \in \mathscr{B}^i_{\ker}$, and the second, if $\{ \langle b \rangle \} \in \mathscr{B}^i_A$.

\textbf{Case 2A}: Assume $\{ \langle b \rangle \} \in \mathscr{B}^i_{\ker}$. We know from Equations (\ref{InducedTotalHomology}) and (\ref{Phi1}) that
\[ \iota^i_{\total} \circ \, \Phi^i (0, \{ \langle b \rangle \}) = \iota^i_{\total}([(-1)^{n+1} \Gamma^i \{ \langle  b \rangle \}, b^* ]) = [(-1)^{n+1} \incv^i_n \circ \Gamma^i \{ \langle b \rangle \}, \ince^i_{n-1}b^*]. \]
On the other hand, from Equations (\ref{PsiExplicit}) and (\ref{Phi1}),
\begin{align*}
\Phi^{i+1} \circ \Psi^i ( 0, \{ \langle b \rangle \}) &= \Phi^{i+1} ((-1)^{n+1} \delta^i \{ \langle b \rangle \} , 0) \\
&= \Phi^{i+1}((-1)^{n+1} \{ \langle -e^{i+1}_n \alpha^{i+1} + \incv^i_n  \circ \Gamma^i \{ \langle b \rangle \}  \, \rangle \}, 0 )\\
&= [ -(-1)^{n+1} e^{i+1}_n \alpha^{i+1} + (-1)^{n+1}\incv^i_n \circ \Gamma^i \{ \langle  b \rangle \},0].
\end{align*}
Then,
\[ \iota^i_{\total} \circ \Phi^i (0, \{ \langle b \rangle \} )- \Phi^{i+1} \circ \Psi^i (0, \{ \langle b \rangle \} )= [(-1)^{n+1}e^{i+1}_n \alpha^{i+1}, \ince^i_{n-1} b^* ]. \]
Recall from Equation (\ref{Alpha}) that $\partial \alpha^{i+1}= \ince^i_{n-1} b^*$. Thus, $\iota^i_{\total} \circ \Phi^i ( 0, \{ \langle b \rangle \} ) - \Phi^{i+1} \circ \Psi^i (0, \{ \langle b \rangle \})$ is trivial in $H_n(\total^{i+1})$, and Diagram \ref{Commute1} commutes for $\{ \langle b \rangle \} \in \mathscr{B}^i_{\ker}$.  

\textbf{Case 2B}: If $\{ \langle b \rangle \} \in \mathscr{B}^i_A$, then again by Equations (\ref{InducedTotalHomology}) and (\ref{Phi1}),
\[ \iota^i_{\total} \circ \, \Phi^i (0, \{ \langle b \rangle \}) = \iota^i_{\total}([(-1)^{n+1} \Gamma^i \{ \langle b \rangle \} , b^* ]) = [(-1)^{n+1} \incv^i_n \circ \Gamma^i \{ \langle b \rangle \}, \ince^i_{n-1} b^* ]. \]

On the other hand, from Equations (\ref{PsiExplicit}) and (\ref{Phi1}),
\begin{align*}
\Phi^{i+1} \circ \Psi^i ( 0, \{ \langle b \rangle \}) &= \Phi^{i+1}(0, \{ \langle \ince^i_{n-1} b \rangle \} ) \\
&= \Phi^{i+1}_1 ( \{ \langle \ince^i_{n-1} b \rangle \} ) \\
&= [(-1)^{n+1} \Gamma^{i+1} \{ \langle \ince^i_{n-1} b \rangle \}, \ince^i_{n-1} b^* ] \\
&= [(-1)^{n+1} \incv^i_n \circ \, \Gamma^i \{ \langle b \rangle \}, \ince^i_{n-1} b^*]
\end{align*}

The last equality follows from the construction of $\Gamma^{i+1}$ in Equation (\ref{GammaOnC}). Thus, the diagram commutes for $\{ \langle b \rangle \} \in \mathscr{B}_A^i$.


Thus, Diagram \ref{DiagramFirst} commutes.
\end{proof}

Lemma \ref{Theorem2} and \ref{Theorem1} yield the desired result.

\begin{Theorem}
\label{MainThm}
The persistence modules $\mathbb{V}_{\Psi}$ and $\mathbb{V}$ are isomorphic.
\end{Theorem}

\section{Application: Multiscale Persistence}
\label{sec:multi}

There are a number of potential uses for distributed persistent homology computations. Perhaps scalable decentralized computation for large data sets is the most obvious: here the partition of the data set into patches is based on localization via coordinates (this is what appears in \cite{CasasDistributing19,LMParallel15}). However, even among small data sets, there are reasons for distributing the computation along a partition of the point cloud based on scalar fields other than coordinates. Data often comes with additional features, such as density estimates, distance to a landmark, and time dependence, that one may wish to examine. In this section, we apply the distributed computation method from \S\ref{sec:DistributedPH} to point cloud data based on density as a parameter. Section \ref{MultiscaleGeneral} provides a general framework for multiscale persistence, and Example \ref{DataDensity} illustrates the advantage of multiscale analysis when examining dataset with varying density. In such situations, multiscale persistent homology allows the user to detect significant features that are overlooked by standard persistent homology methods.

\subsection{Multiscale Barcode Annotation}
\label{MultiscaleGeneral}

We provide a general framework for computing multiscale persistent homology. Given a point cloud $P$, let $f: P \to \mathbb{R}$ be a (user-chosen)
density estimate. Construct a cover $\mathscr{V}$ of $f(P)$
with nerve $N_{\mathscr{V}}$ a compact 
interval.
Let $\mathbb{V}$ denote the persistence module
\begin{equation}
\label{UsualPH}
\mathbb{V} : H_n(\mathscr{R}^1) \to H_n(\mathscr{R}^2) \to \dots \to H_n(\mathscr{R}^N)
\end{equation}
in the usual sense. Let $\barcode(\mathbb{V})$ denote the barcode of $\mathbb{V}$. If a bar of the barcode represents a feature $\gamma$ that consists of points in $f^{-1}(U)$ for some $U \in \mathscr{V}$, we say that the feature $\gamma$ lives in $U$. Moreover, we can annotate the bar with its corresponding set $U$. The goal of multiscale persistent homology is to annotate the bars of $\barcode(\mathbb{V})$ with their corresponding sets $U$ of $\mathscr{V}$.

An algorithmic summary of the annotation process is provided, followed by a detailed explanation of each step.

\begin{algorithm}
\caption{Annotate $\barcode(\mathbb{V})$.}
\label{TheAlgorithm}
\begin{algorithmic}[1]
    \State Compute persistence module $\mathbb{V}_*$ using distributed computation.
    \State Label vector spaces of $\mathbb{V}_*$.
    \State  For each persistence module $\mathbb{W}_s$ of $\mathbb{V}_*=\bigoplus\limits_s \mathbb{W}_s$, annotate $\barcode(\mathbb{W}_s)$.
    \State Using the annotations of $\barcode(\mathbb{V}_*)$, annotate $\barcode(\mathbb{V})$.
    \State Return annotated $\barcode(\mathbb{V})$.
\end{algorithmic}
\end{algorithm}

\subsubsection*{Step 1. Compute persistence module $\mathbb{V}_*$}

Let $\mathbb{V}$ denote the persistence module of interest
\begin{equation}
\label{UsualPH2}
\mathbb{V} : H_n(\mathscr{R}^1) \to H_n(\mathscr{R}^2) \to \dots \to H_n(\mathscr{R}^N).
\end{equation}

Let $\epsilon_*$ be the upper bound from Lemma \ref{DiscreteCGNBound} such that
\[H_n(\mathscr{R}^{\epsilon}) \cong H_0(C_{\bullet} \cosheaf{F}^{\epsilon}_{n}) \oplus H_1(C_{\bullet} \cosheaf{F}^{\epsilon}_{n-1})\]
for all $\epsilon < \epsilon_*$. Let $\mathbb{V}|_L$ denote the sequence of vector spaces and maps of $\mathbb{V}$ up to parameter $\epsilon_L $ such that $\epsilon_L< \epsilon_*$:
\[\mathbb{V} |_L : H_n(\mathscr{R}^1) \to H_n(\mathscr{R}^2) \to \dots \to H_n(\mathscr{R}^L).\]
We can compute the persistence module
\begin{equation}
\label{TruncatedDistributed}
\mathbb{V}_{\Psi} : H_0(C_{\bullet} \cosheaf{F}^1_n)\oplus H_1(C_{\bullet} \cosheaf{F}^1_{n-1}) \xrightarrow{\Psi^1} \dots \xrightarrow{\Psi^{L-1}} H_0(C_{\bullet} \cosheaf{F}^L_n)\oplus H_1(C_{\bullet} \cosheaf{F}^L_{n-1})
\end{equation}
isomorphic to $\mathbb{V}|_L$ using the distributed computation method from \S\ref{sec:DistributedPH}.

We will in fact compute a persistence module $\mathbb{V}_*\cong\mathbb{V}_{\Psi}$ that can reveal additional information about the barcode. Recall from \S\ref{sec:CellularCosheaves} that each cosheaf $\cosheaf{F}^{i}_n$ can be decomposed as $\cosheaf{F}^{i}_n \cong \oplus \cosheaf{I}^{i}_n$, where each $\cosheaf{I}^i_n$ is an indecomposable cosheaf over $N_{\mathscr{V}}$. In other words, there exists an isomorphism of cosheaves
\begin{equation}
\label{IsoCosheafMorphism}
D^i_n : \cosheaf{F}^i_n \to \oplus \cosheaf{I}^i_n.
\end{equation}

For each parameter $\epsilon_i$, there exists an isomorphism 
\[g^i: H_0(C_{\bullet} \cosheaf{F}^i_n)\oplus H_1(C_{\bullet} \cosheaf{F}^i_{n-1}) \to H_0(C_{\bullet} \oplus \cosheaf{I}^i_n) \oplus H_1(C_{\bullet} \cosheaf{F}^i_{n-1})\]
defined by
\[ g^i ( \, \{ \langle x \rangle \}, \{ \langle y \rangle \} \,) = (\, H_0(D^i_n) \{ \langle x \rangle \}, \{ \langle y \rangle \} \, ),\]
where $H_0(D^i_n) : H_0(C_{\bullet} \cosheaf{F}^i_n) \to H_0(C_{\bullet} \oplus \cosheaf{I}^i_n)$ is the isomorphism induced by $D^i_n$.

Define the persistence module $\mathbb{V}_*$ by
\[\mathbb{V}_* : H_0(C_{\bullet} \oplus \cosheaf{I}^{1}_n) \oplus H_1(C_{\bullet} \cosheaf{F}^{1}_{n-1}) \xrightarrow{\Psi^1_*} \cdots \xrightarrow{\Psi^{N-1}_*} H_0(C_{\bullet} \oplus \cosheaf{I}^{N}_n) \oplus H_1(C_{\bullet} \cosheaf{F}^{N}_{n-1}), \]
where the map $\Psi^i_*$ is defined by $\Psi^i_* = g^{i+1} \circ \Psi^i \circ (g^i)^{-1}$.

By construction, $\mathbb{V}_*$ is isomorphic to the persistence module $\mathbb{V}_{\Psi}$ and hence isomorphic to $\mathbb{V}|_L$. Our mechanism of decomposing the cosheaf $\cosheaf{F}^{i}_n$ into indecomposable cosheaves may seem like a cumbersome step. However, such decomposition allows us to understand the cosheaf homologies in terms of the indecomposable cosheaves $\cosheaf{I}^i_n$.

\subsubsection*{Step 2. Label the vector spaces of $\mathbb{V}_*$}

For any parameter $\epsilon_i$, recall that
\[\mathbb{V}_*^{i} = H_0(C_{\bullet} \oplus \cosheaf{I}^{i}_n) \oplus H_1(C_{\bullet} \cosheaf{F}^{i}_{n-1}).\]
By Lemma \ref{IndecomposableCosheaf}, each component of $H_0(C_{\bullet} \oplus \cosheaf{I}^i_n)$ corresponds to an indecomposable cosheaf of the form $\cosheaf{I}^i_{[-]}$.
We can thus annotate each component of $H_0(C_{\bullet} \oplus \cosheaf{I}^i_n)$ according the support of the corresponding indecomposable $\cosheaf{I}^i_{[-]}$.

\begin{enumerate}[label = \textit{Case \arabic*}: , leftmargin=*, align=left]
\item If the indecomposable $\cosheaf{I}^i_{[-]}$ is supported on a unique vertex $v \in N_{\mathscr{V}}$, then annotate the component by $U \in \mathscr{V}$, where $U$ is the open set corresponding to the vertex $v$.
\item Let $v_l, v_r \in N_{\mathscr{V}}$ each denote the left and rightmost supports of $\cosheaf{I}^i_{[-]}$. If $v_l, v_{l+1}, \dots, v_r$ are the verices of $N_{\mathscr{V}}$ between $v_l$ and $v_r$, then the cosheaf $\cosheaf{I}^i_{[-]}$ represents a feature that lives in all $U_l, U_{l+1}, \dots, U_r \in \mathscr{V}$. The user can annotate the corresponding component by $[U_l, U_r]$ or $U_l$ or $U_r$, depending on the user's goal.
\end{enumerate}

For example, assume that
\[ \mathbb{V}_*^{i} = H_0(C_{\bullet} \oplus \cosheaf{I}^{i}_{n}) \oplus H_1(C_{\bullet} \cosheaf{F}^{i}_{n-1})= \mathbb{K} \oplus \mathbb{K} \oplus \mathbb{K} \oplus \mathbb{K},\]
where the first three components come from $H_0(C_{\bullet}\oplus \cosheaf{I}^{i}_{n})$ and the last component $\field{K}$ comes from $H_1(C_{\bullet} \cosheaf{F}^{i}_{n-1})$. An example of cosheaf $\oplus \cosheaf{I}^{i}_{n}$ is illustrated in Figure \ref{fig:example_decomposition}.

\begin{figure}[h!]
\centering
\includegraphics[scale=0.25]{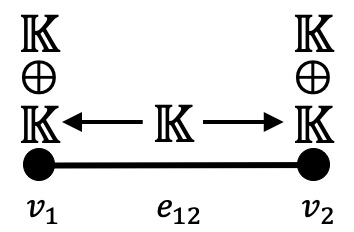}
\caption{An example decomposition of a cosheaf $\cosheaf{F}^i_{n} \cong \oplus\cosheaf{I}^i_n$}
\label{fig:example_decomposition}
\end{figure}

Then one can label the components of $H_0(C_{\bullet} \oplus \cosheaf{I}^{i}_{n})$ by $\mathbb{K}_1 \oplus \mathbb{K}_2 \oplus \mathbb{K}_{1,2}$, where each label corresponds to the support of the indecomposable cosheaf in Figure \ref{fig:example_decomposition}. Then, the vector space $\mathbb{V}^{i}_*$ can be labeled as $\mathbb{K}_1 \oplus \mathbb{K}_2 \oplus \mathbb{K}_{1,2} \oplus \mathbb{K}$.

\subsubsection*{Step 3. Annotate the $\barcode(\mathbb{W}_s)$ for each $\mathbb{W}_s$ of $\mathbb{V}_*=\bigoplus\limits_s \mathbb{W}_s$}
Note that $\mathbb{V}_*$ can be expressed naturally as a sum of persistence modules as
\[\mathbb{V}_* = \bigoplus\limits_s \mathbb{W}_s.\]

For example, the persistence module
\[\mathbb{V}_*: \mathbb{R}^3 \xrightarrow{ \begin{bmatrix}  1 & 0 &  0 \\ 1 & 0 & 0 \\ 0 & 1 & 1 \end{bmatrix}} \mathbb{R}^3 \xrightarrow{ \begin{bmatrix}  1 & -1 &  0 \\ 0 & 0 & 1 \\ 0 & 0 & 1 \end{bmatrix}} \mathbb{R}^3\]
is the sum of persistence modules
\[ \mathbb{W}_1: \mathbb{R} \xrightarrow{ \begin{bmatrix}  1 \\ 1   \end{bmatrix}} \mathbb{R}^2 \xrightarrow{ \begin{bmatrix}  1 & -1  \end{bmatrix}} \mathbb{R} \quad \textrm{and} \quad \mathbb{W}_2: \mathbb{R}^2 \xrightarrow{ \begin{bmatrix} 1 & 1 \end{bmatrix}} \mathbb{R} \xrightarrow{ \begin{bmatrix}   1 \\  1 \end{bmatrix}} \mathbb{R}^2.\]

Moreover, $\barcode(\mathbb{V}_*)$ is the collection of barcodes $\barcode(\mathbb{W}_s)$. For each $\mathbb{W}_s$, compute $\barcode(\mathbb{W}_s)$. Let $\overline{b}$ be a bar of $\barcode(\mathbb{W}_s)$ born at parameter $\epsilon_i$. Consider the annotation of the components of $\mathbb{W}_s^i$ from Step 2. There are two cases to consider.

\begin{enumerate}[label = \textit{Case \arabic*}: , leftmargin=*, align=left]
\item All components of $\mathbb{W}^i_s$ have been annotated by a unique set $U \in \mathscr{V}$. Bar $\overline{b}$ then represents some linear combination of features in $U$, so annotate $\overline{b}$ with $U$.
\item The components of $\mathbb{W}^i_s$ have been annotated by two or more sets in $\mathscr{V}$, say $U_j$ and $U_k$. The user can decide to either not annotate the bar at all, to annotate the bar by $U_j$, or to annotate the bar by $U_k$, depending on the question of interest.
\end{enumerate}

The result of Step 3 is an annotation of $\barcode(\mathbb{V}_*)$.

\subsubsection*{Step 4. Annotate $\barcode(\mathbb{V}$)}  
We can use the annotations of $\barcode(\mathbb{V}_*)$ to annotate $\barcode(\mathbb{V})$. Note that $\barcode(\mathbb{V}_*)$ can be obtained from $\barcode(\mathbb{V})$ by truncating $\barcode(\mathbb{V})$ at parameter $\epsilon_L$. Let $[b,d]$ be a bar of $\barcode(\mathbb{V}_*)$ that has been annotated by a set $U$ in Step 3. There are two cases to consider.
\begin{enumerate}[label = \textit{Case \arabic*}: , leftmargin=*, align=left]
\item If $d< \epsilon_L$, then annotate the bar $[b,d]$ of $\barcode(\mathbb{V})$ with $U$. 
\item If $d = \epsilon_L$ and $[b,d]$ is the unique bar of $\barcode(\mathbb{V}_*)$ with birth time $b$, then there exists a unique bar $[b, d']$ in $\barcode(\mathbb{V})$ with birth time $b$. Annotate the bar $[b,d']$ of $\barcode(\mathbb{V})$ with $U$.
\end{enumerate}

The final result of the algorithm is an annotation of $\barcode(\mathbb{V})$. One can use this annotated barcode to perform finer data analysis.

\subsection{Example: Multiscale Persistence}
\label{DataDensity}

Consider a situation where the size of a feature depends on the density of the constituting points, as illustrated in Figure \ref{fig:PointCloudExample}. Figure \ref{TotalBarcode} illustrates the corresponding barcode, which suggests that there is one significant feature. Standard persistent homology fails to detect the small, but densely sampled features. Multiscale persistent homology can select the bars that correspond to small but densely sampled features and annotate them as being significant.

\begin{figure}[h!]
\centering
\includegraphics[scale=0.25]{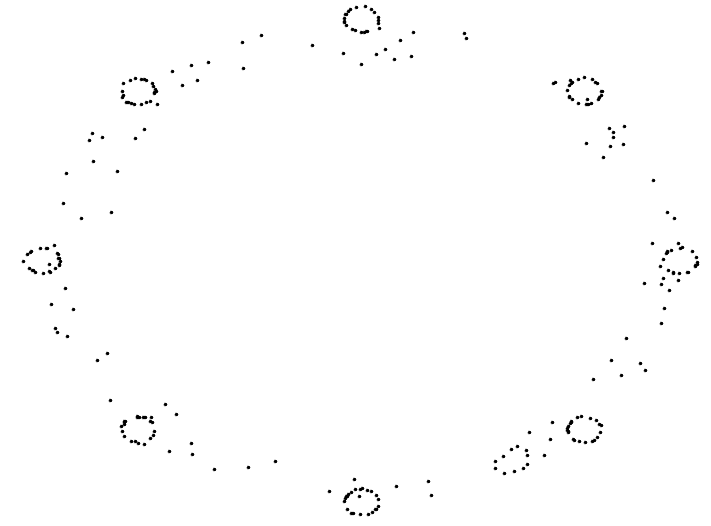}
\caption{A point cloud with varying density}
\label{fig:PointCloudExample}
\end{figure}
\begin{figure}[h!]
\centering
\includegraphics[scale=0.3]{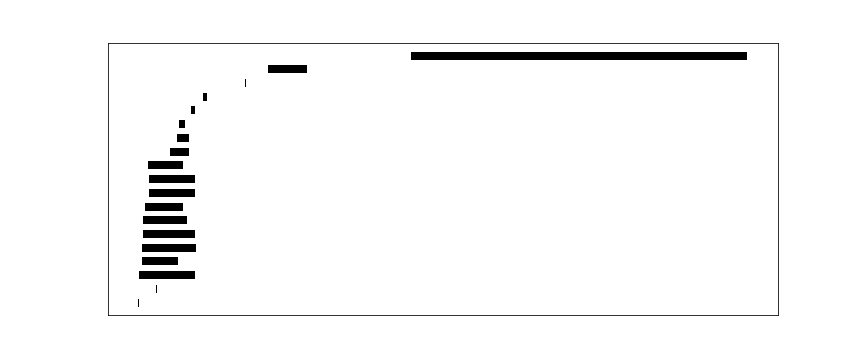}
\caption{Barcode from standard persistent homology in dimension $1$}
\label{TotalBarcode}
\end{figure}

Let $P$ denote the point cloud in Figure \ref{fig:PointCloudExample}, and let $f : P \to \mathbb{R}$ be the function mapping each point to its estimated density value. In our example, $f(p)$ represents the number of points in a radius $r$-ball centered at $p$.

We chose a covering $\mathscr{V}=\{U_s, U_d \}$ of $f(P)$ by the following. We first plotted a histogram of density values as illustrated in Figure \ref{histogramPlot}, and decided to cover $f(P)$ with two sets, $U_s$ and $U_d$, where $U_s = (0,18)$ and $U_d=(8,26)$. We will refer to points in $f^{-1}(U_s)$ as the sparse points, and we will refer to points in $f^{-1}(U_d)$ as the dense points. Figures \ref{fig:SparsePoints} and \ref{fig:DensePoints} illustrate the sparse and dense points.
\begin{figure}[h!]
\centering
\includegraphics[scale=0.18]{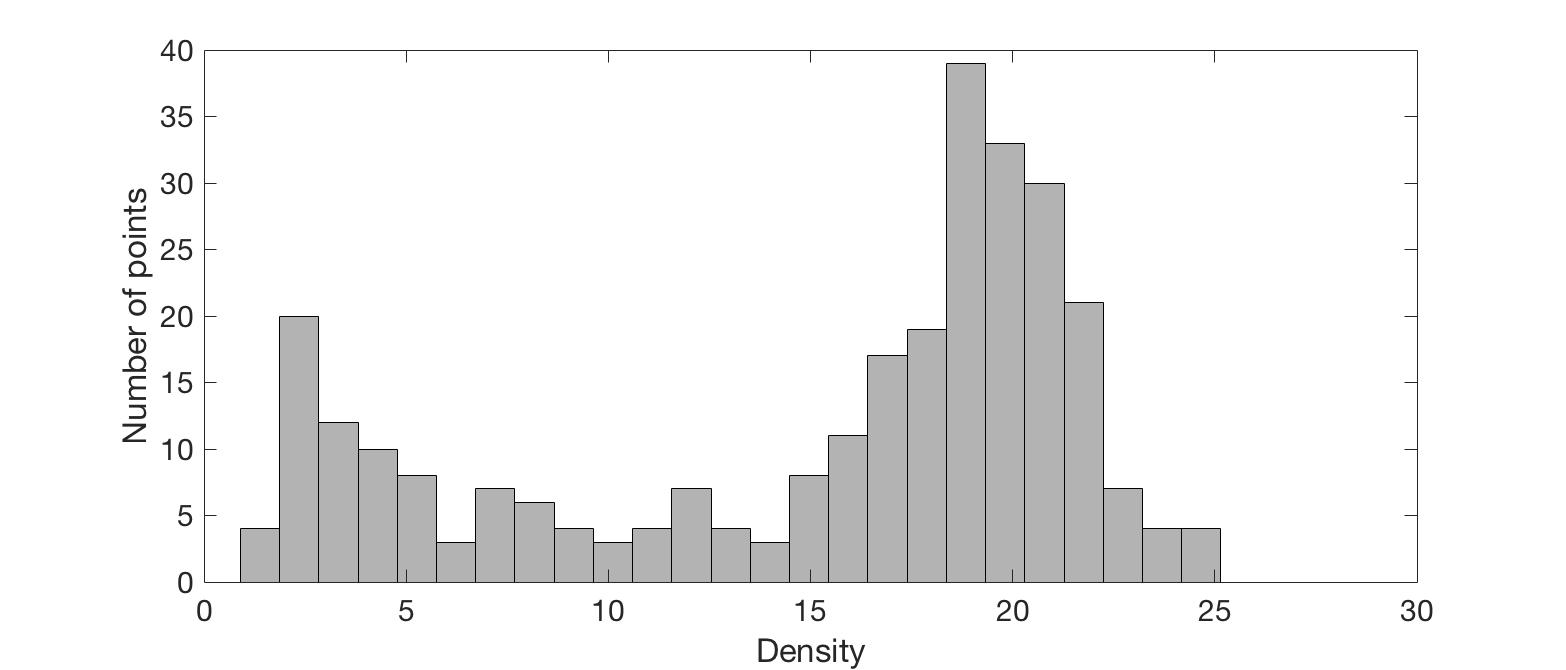}
\caption{Histogram plot of estimated density values.}
\label{histogramPlot}
\end{figure}

\begin{figure}[h!]
\centering
\begin{subfigure}[b]{0.48\textwidth}
\includegraphics[scale=0.125]{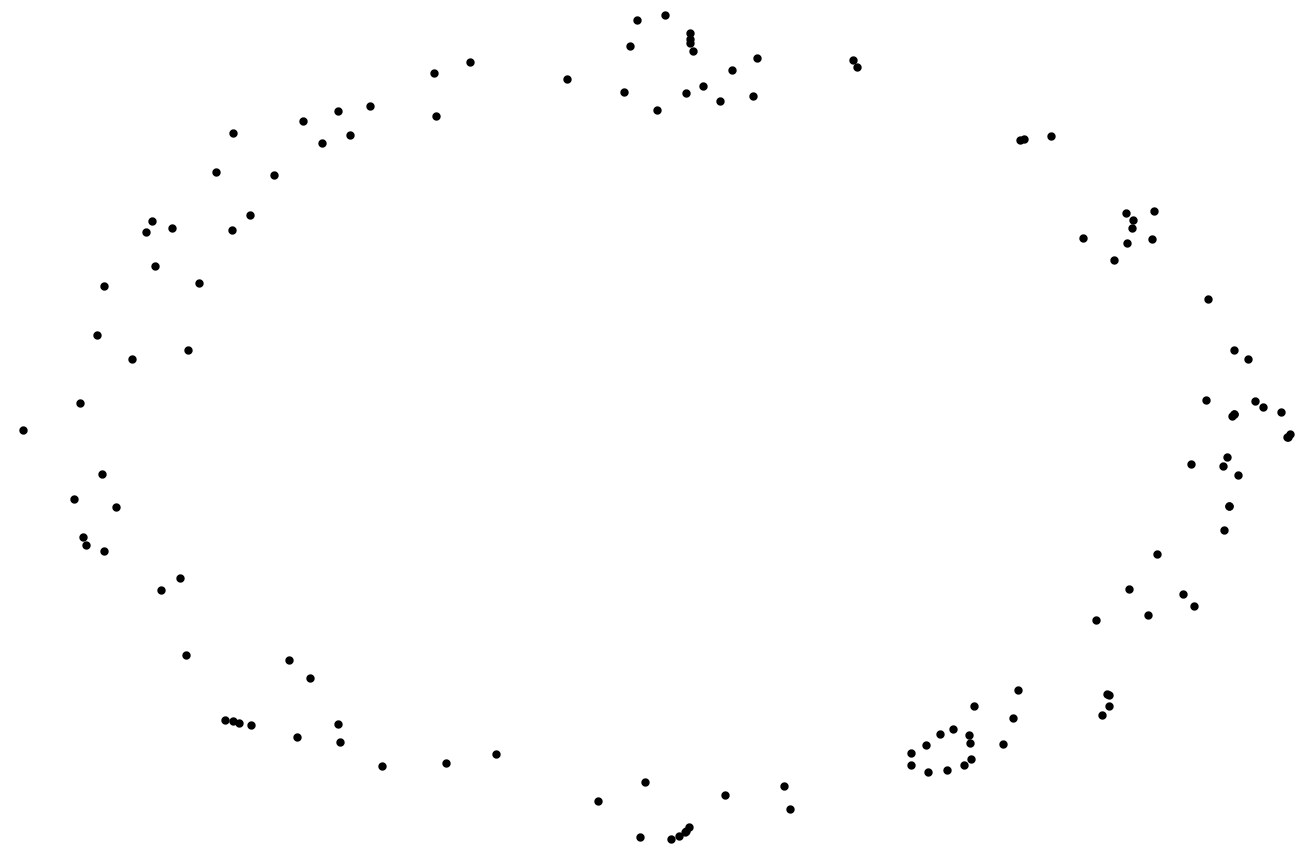}
\caption{Sparse points}\label{fig:SparsePoints}
\end{subfigure}
\begin{subfigure}[b]{0.48\textwidth}
\includegraphics[scale=0.125]{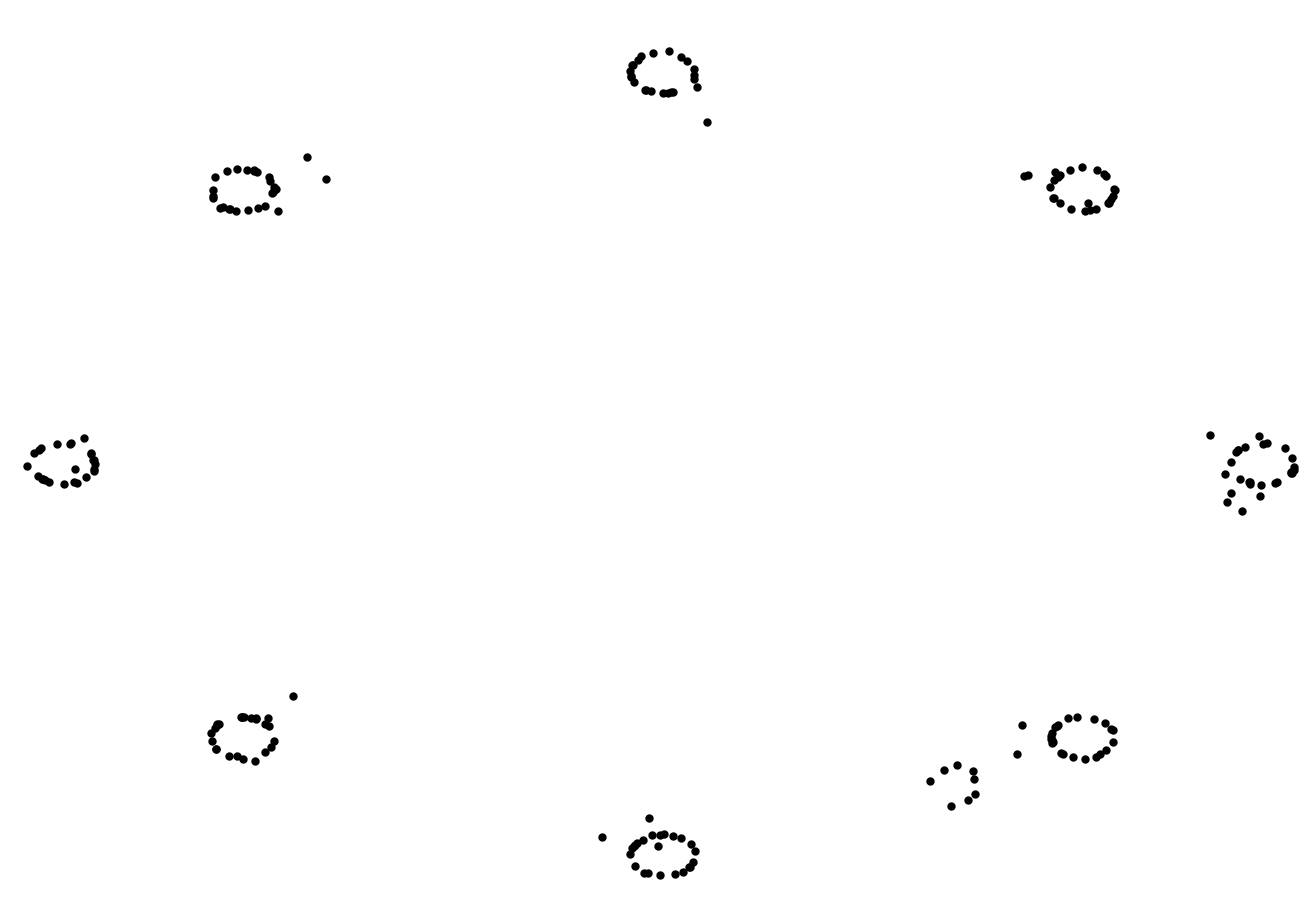}
\caption{Dense points}\label{fig:DensePoints}
\end{subfigure}
\caption{Sparse and dense points.}\label{fig:SparseDensePoints}
\end{figure}

We now follow Algorithm \ref{TheAlgorithm}. Let
\[ \mathbb{V} : H_1(\mathscr{R}^1) \to \dots \to H_1(\mathscr{R}^N) \]
be the persistence module obtained from the point cloud $P$. For this example, the maximum parameter is $\epsilon_N = 1.6$. Let $\epsilon_*$ be the upper bound of the parameter $\epsilon$ from Lemma \ref{DiscreteCGNBound} for which the isomorphism
\[H_n(\mathscr{R}^{\epsilon}) \cong H_0(C_{\bullet} \cosheaf{F}^{\epsilon}_n) \oplus H_1( C_{\bullet} \cosheaf{F}^{\epsilon}_{n-1}) \]
holds. For this example, $\epsilon_*$ is $0.0719$. Compute the persistence module
\begin{equation}
\label{DistributedTruncated}
\mathbb{V}_* : H_0(C_{\bullet}\oplus \cosheaf{I}^{1}_1) \oplus H_1(C_{\bullet} \cosheaf{F}^{1}_{0}) \to \dots \to
H_0(C_{\bullet} \oplus \cosheaf{I}^{\epsilon_*}_1) \oplus H_1(C_{\bullet}\cosheaf{F}^{\epsilon_*}_0)
\end{equation}
following Step 1 of Algorithm \ref{TheAlgorithm}.

Step 2 of Algorithm \ref{TheAlgorithm} labels the components of vector space $H_0(C_{\bullet} \oplus \cosheaf{I}^i_1)$ according to the support of the indecomposable cosheaves $\cosheaf{I}^i_{[-]}$.

Step 3 of Algorithm \ref{TheAlgorithm} results in an annotated version of $\barcode(\mathbb{V}_*)$, illustrated in Figure \ref{TruncatedBarcodeAnnotated}. The top two gray bars have been annotated by $U_s$. The two bars represent features in the sparse points. The remaining black bars have been annotated by $U_d$, and they represent features in the dense points.

\begin{figure}[h!]
\centering
\includegraphics[scale=0.3]{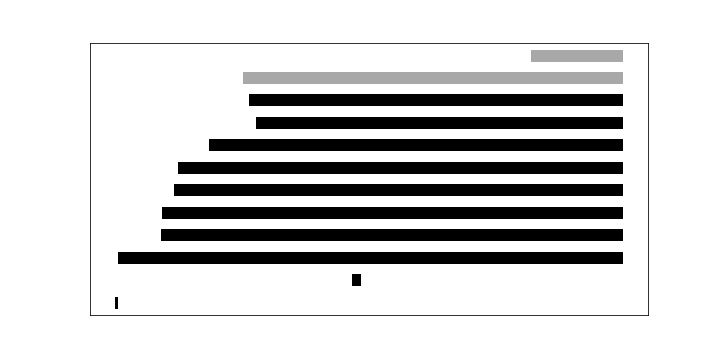}
\caption{Annotated $\barcode(\mathbb{V}_*)$. The top two gray bars have been annotated by $U_s$, and they represent features in the sparse points. The remaining black bars have been annotated by $U_d$, and they represent features in the dense points.}
\label{TruncatedBarcodeAnnotated}
\end{figure}
Step 4 of Algorithm \ref{TheAlgorithm} allows us to transfer the annotation of $\barcode(\mathbb{V}_*)$ to $\barcode(\mathbb{V})$ resulting in an annotated version of $\barcode(\mathbb{V})$ illustrated in Figure \ref{TotalBarcodeAnnotated}. The two bars enclosed by the gray box are annotated by $U_s$, and the bars enclosed by black box are annotated by $U_d$.
\begin{figure}[h!]
\centering
\includegraphics[scale=0.4]{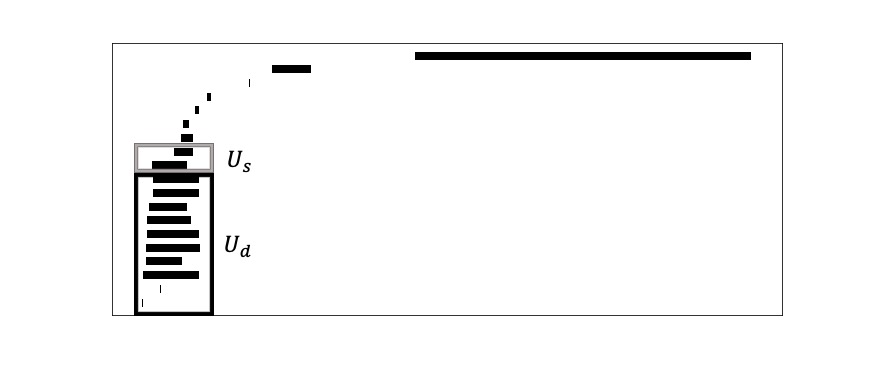}
\caption{Annotated $\barcode(\mathbb{V})$. The two bars enclosed by the gray box are annotated by $U_s$, and the bars  dyenclosed by black box are annotated by $U_d$.}
\label{TotalBarcodeAnnotated}
\end{figure}

The goal is to determine the small but significant features that consist of the denser points. Thus, we focus on the bars of Figure \ref{TotalBarcodeAnnotated} that have been annotated by $U_d$. By restricting our attention to only the bars that represent features in $U_d$, we are able to determine the significant features built among the dense points. By examining the bars annotated by $U_d$, one can conclude that there are eight significant bars.

Lastly, we return to $\barcode(\mathbb{V})$ and indicate the significant features. We then obtain the barcode in Figure \ref{AnnotatedBarcode}, where the black bars indicate significant features and the gray bars indicate noise. Note that we have one long black bar, which is deemed significant because of its length. We have eight additional shorter significant bars which were identified via Algorithm \ref{TheAlgorithm}.

\begin{figure}[h!]
\centering
\includegraphics[scale=0.4]{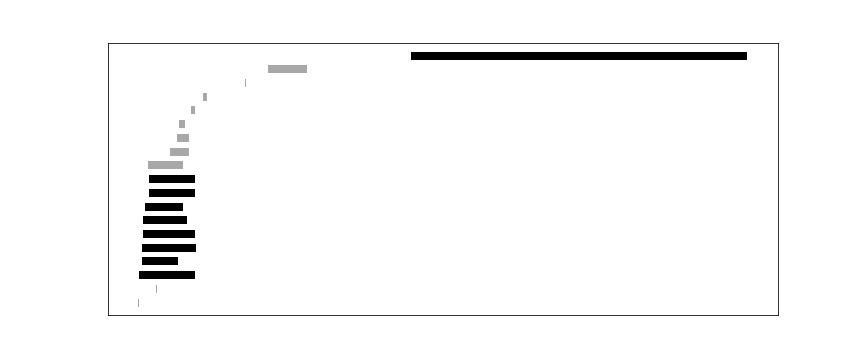}
\caption{Final annotation of $\barcode(\mathbb{V})$}
\label{AnnotatedBarcode}
\end{figure}

The persistent homology computation Julia package {\em Eirene.jl} \cite{HenselmanGhrist16} identifies persistent homology generators in the point cloud. Using this, we identified the points of $P$ that constitute each significant feature. The eight significant short bars identified indeed correspond to the eight small but densely sampled features in Figure \ref{fig:PointCloudExample}.



\begin{appendices}
\section{Proof of Lemma \ref{DiscreteCGNBound}}
\label{A:DiscreteCGNBound}

\discreteCGNlemma*

\begin{proof}
We first specify $\epsilon_*$. In what follows, use the convention that the minimum over an empty set is $\infty$.

Each $p \in P$ lies in either one or two elements of the cover $\mathscr{V}$. If unique, $p\in U \in \mathscr{V}$, then set
\begin{equation}
\label{KpEq1}
K_p = \min_{\{ q \notin U \}} d(p,q).
\end{equation}
If $p$ lies in two sets of the cover,  $p\in U\cap W$, then first let
\begin{equation}
\label{kp1}
K_p' =\min_{ \{ q \notin U \cup W \} } d(p,q),
\end{equation}
and let
\begin{equation}
\label{kp2}
K_p'' = \min_{ \substack{ \{ q, q' |  d(p,q) < K_p', \quad q\notin U, \\
			  \qquad d(p,q')<K_p', \quad q' \notin W \} } }
			  d(q, q').
\end{equation}
Then set
\[ K_p = \min \{ K_p', K_p'' \}. \]
Let $\epsilon_* = \min_{p \in P} K_p$. Assume $\epsilon < \epsilon_*$. We assert Equation (\ref{DistributedComputationK}) by showing that the Rips system covers $\mathscr{R}^{\epsilon}$. Let $\omega = (v_0, \dots, v_l)$ be a simplex of $\mathscr{R}^{\epsilon}$ with vertices as listed. The pairwise distances thus satisfy $d(v_i, v_j) < \epsilon$.

If there exists a vertex of $\omega$, say $v_0$, such that $v_0$ belongs to a unique $U\in\mathscr{V}$, then by construction, $\epsilon_* < K_{v_0} = \min_{\{ q \notin U \}} d(v_0, q)$ from Equation (\ref{KpEq1}). Thus, for any other vertex $v$ of $\omega$, we have $d(v_0,v) < \epsilon < K_{v_0}$, and hence all of $\omega$ is covered by $U$.

Otherwise, every vertex $v$ of $\omega$ is covered by two sets in $\mathscr{V}$. Without loss of generality, assume that $v_0 \in U \cap W$. Note that for any other vertex $v$ of $\omega$, we have
\begin{equation}
\label{ineq}
d(v_0, v) < \epsilon < K_{v_0} \leq K_{v_0}',
\end{equation}
where $K_{v_0}'$ is given by Equation (\ref{kp1}). So $v\in U\cup W$ for every $v \in \omega$. In fact, we can show that all of $\omega$ lies in $U$ or in $W$. Assuming the contrary, there exist distinct vertices, say $v_1$ and $v_2$, such that $v_1 \notin U$ and $v_2 \notin W$. By construction, $d(v_0, v_1) < \epsilon < K_{v_0}'$, and $d(v_0, v_2) < \epsilon < K_{v_0}'$. By definition of $K_{v_0}''$ from Equation (\ref{kp2}), we know that $K_{v_0}'' \geq d(v_1, v_2)$. However, this contradicts the fact that $d(v_1, v_2) < \epsilon < K_{v_0}''$. Thus, $\omega$ is covered by some subcomplex $\mathscr{R}^{\epsilon}_{\sigma}$. Thus, the Rips system covers $\mathscr{R}^{\epsilon}$, and Lemma \ref{DiscreteCGNBound} follows from the proof of the analogous result in \cite{CGNDiscrete13}.
\end{proof}

\section{Independence of $\alpha^{i+1}$}
\label{DeltaWellDefined}

\begin{Lemma}
The construction of $\delta^i$ on a basis element $\{ \langle b \rangle \} \in \mathscr{B}^i_{\ker}$ in Equation (\ref{deltai_def_basis}) does not depend on the choice of $\alpha^{i+1}$.
\end{Lemma}
\begin{proof}
Let $\alpha_1, \alpha_2 \in \bigoplus\limits_{e \in N_{\mathscr{V}}} C_n(\mathscr{R}^{i+1}_e)$ be two different choices of $\alpha$ that satisfy Equation (\ref{Alpha}):
$\partial \alpha = \ince^i_{n-1} b^*$.
Note that $\partial (\alpha_1-\alpha_2) =\ince^i_{n-1} b^*-\ince^i_{n-1} b^*=0$. So $\langle \alpha_1 - \alpha_2 \rangle $ represents an element in $\bigoplus\limits_{e \in N_{\mathscr{V}}} H_n(\mathscr{R}^{i+1}_e)$. Then, $ \langle e^{i+1}_n(\alpha_1-\alpha_2) \rangle =  \partial^{i+1}_n \langle \alpha_1-\alpha_2 \rangle \in \im \partial^{i+1}_n$. So $\{ \langle e^{i+1}_n (\alpha_1 -\alpha_2) \rangle  \}$ is trivial in $H_0(C_{\bullet} \cosheaf{F}^{i+1}_n)$. Hence, $ \{ \langle- e^{i+1}_n \alpha_1+ \incv^i_n \circ \beta^i \rangle \}= \{ \langle -e^{i+1}_n \alpha_2+\incv^i_n \circ \beta^i  \rangle \}$ in $H_0(C_{\bullet} \cosheaf{F}^{i+1}_n)$. Thus, the map $\delta^i$ does not depend on the choice of $\alpha$.
\end{proof}

\section{Obtaining basis $\mathscr{C}^i_{\im}$ from basis $\mathscr{B}^{i-1}_A$ of $A^{i-1}$}
\label{LinearIndependence}

\begin{Lemma}
Let $\mathscr{B}^{i-1}= \mathscr{B}^{i-1}_A \cup \mathscr{B}^{i-1}_{\ker}$ be a basis of $H_1(C_{\bullet} \cosheaf{F}^{i-1}_{n-1}) =A^{i-1} \oplus \ker H_1(\phi^{i-1}_{n-1})$. Let $\mathscr{B}^{i-1}_A = \{ \, \{ \langle b_1 \rangle \}, \dots, \{ \langle b_t \rangle \} \, \}$ be the basis of $A^{i-1}$. Then,
\[
\{ \langle \ince^{i-1}_{n-1} b_1 \rangle \} , \dots, \{ \langle \ince^{i-1}_{n-1} b_t \rangle \}
\]
are linearly independent in $H_1(C_{\bullet} \cosheaf{F}^i_{n-1})$.
\end{Lemma}

\begin{proof}
Assume the contrary, that
\[ c_1 \{ \langle \ince^{i-1}_{n-1} b_1 \rangle \} + \cdots + c_t \{ \langle \ince^{i-1}_{n-1} b_t \rangle \} = \{ \langle 0 \rangle \} \]
for some $c_1, \dots, c_t$ that are not all zero. By construction, this implies that
\[ c_1 \langle \ince^{i-1}_{n-1} b_1 \rangle + \cdots + c_t \langle \ince^{i-1}_{n-1} b_t \rangle = \langle 0 \rangle \]
for some $c_1, \dots, c_t$ that are not all zero. Then, $\langle c_1 b_1 + \dots c_t b_t \rangle \in \ker H_1(\phi^{i-1}_{n-1})$. Note that $\langle c_1 b_1 + \dots c_t b_t \rangle \in A^{i-1}$ as well since $A^{i-1}$ is a subspace of $H_1(C_{\bullet} \cosheaf{F}^{i-1}_{n-1})$. This contradicts the fact that $H_1(C_{\bullet} \cosheaf{F}^{i-1}_{n-1})$ is a direct sum of $\ker H_1^{i-1}(\phi_{n-1})$ and $A^{i-1}$.
Thus, $\{ \langle \ince^{i-1}_{n-1} b_1 \rangle \} , \dots, \{ \langle \ince^{i-1}_{n-1} b_t \rangle \}$ are linearly independent.
\end{proof}

\section{Details of proof of Lemma \ref{Theorem2}}
\label{ProofTheorem2}
\begin{Lemma}
The map $\Phi^i_{\total} : H_n(\total^i) \to H_n(\mathscr{R}^i) $ is well-defined and bijective.
\end{Lemma}
For clarity, we omit the superscript $i$ indicating the parameter $\epsilon_i$.
\begin{proof}
We first show that $\Phi_{\total}: H_n(\total) \to H_n(\mathscr{R})$ is a well-defined map. Assume that $[x_1, y_1] = [x_2,y_2]$ in $H_n(\total)$. One can build a  commutative diagram whose rows are short exact sequences
\begin{equation}
\label{DoubleComplexSection}
\begin{tikzcd}
0 &\arrow{l} C_{n}(\mathscr{R}) & \arrow{l}{j_{n}} \bigoplus\limits_{v \in N_{\mathscr{V}}} C_{n}(\mathscr{R}_v)  & \arrow{l}{e_{n}} \bigoplus\limits_{e \in N_{\mathscr{V}}} C_{n}(\mathscr{R}_e)  & \arrow{l} 0
\end{tikzcd}
\end{equation}
and whose vertical maps are the boundary operator $\partial_n$.
Since $[x_1, y_1] =[x_2, y_2]$ in $H_n(\total)$, there exists $p_{n+1} \in \bigoplus \limits_{v \in N_{\mathscr{V}}} C_{n+1} (\mathscr{R}_v)$ and $q_n \in \bigoplus\limits_{e \in N_{\mathscr{V}}} C_n(\mathscr{R}_e)$ such that $y_2  - y_1 = \partial q_n$ and $x_2-x_1 = \partial p_{n+1} +(-1)^{n+1} e_n q$. Then,
\begin{align*}
j_n(x_2-x_1) =& j_n( \partial p_{n+1} + (-1)^{n+1}e_n q)  \\
=& j_n (\partial p_{n+1}) \\
=& \partial (j_{n+1}(p_{n+1})).
\end{align*}
The second equality follows from the exactness of Diagram \ref{DoubleComplexSection}. Thus, $[j_n(x_2)] = [j_n(x_1)]$, and the map $\Phi_{\total}$ is well-defined.

We now show that $\Phi_{\total}$ is surjective. Let $[\gamma] \in H_n(\mathscr{R})$. Since the rows of Diagram \ref{DoubleComplexSection} are exact, there exists $x_n \in \bigoplus\limits_{v \in N_{\mathscr{V}}} C_n(\mathscr{R}_v)$ such that $\gamma = j_n (x_n)$. Then,
\[ j_{n-1} \circ \partial x_n = \partial \circ j_n ( x_n) = \partial \gamma = 0.\]
So $\partial x_n \in \ker j_{n-1}$. By exactness, there exists $y_{n-1} \in \bigoplus\limits_{e \in N_{\mathscr{V}}} C_{n-1}(\mathscr{R}_e)$ such that $\partial x_n = e_{n-1} ( y_{n-1})$. Moreover,
\[ e_{n-2} \circ \partial y_{n-1} = \partial \circ e_{n-1} y_{n-1} = \partial \partial x_n = 0.\]
Since $e_{n-2}$ is injective, we know that $\partial y_{n-1}=0$. Then, $[x_n, (-1)^{n-1} y_{n-1}] \in H_n(\total)$, and $\Psi_{\total}[x_n, (-1)^{n-1} y_{n-1}] = [\gamma]$. Thus, $\Psi_{\total}$ is surjective.

Lastly, we show that $\Phi_{\total}$ is injective. Assume that $\Phi_{\total}([x,y])=[j_n(x)] = 0$. Then, there exists $p_{n+1} \in C_{n+1}(\mathscr{R})$ such that $\partial p_{n+1} = j_n(x)$. Since $j_n$ and $j_{n+1}$ are surjective, there exists $p_{n+1}' \in \bigoplus\limits_{v \in N_{\mathscr{V}}} C_{n+1}(\mathscr{R}_v)$ such that $p_{n+1} =j_{n+1} (p_{n+1}')$. Then,
\begin{align*}
j_n(\partial p'_{n+1}-x) &= j_n \circ \partial p'_{n+1} - j_n(x) \\
&= \partial \circ j_{n+1}(p'_{n+1}) -j_n(x) \\
&= \partial p_{n+1} - j_n(x) = 0. 
\end{align*}

Thus, $\partial p'_{n+1} - x \in \ker j_n$. From the exactness of rows of Diagram \ref{DoubleComplexSection}, there exists $q_n \in \bigoplus\limits_{e \in N_{\mathscr{V}}} C_n(\mathscr{R}_e)$ such that $\partial p'_{n+1} - x = e_n( q_n)$.
Note that while
$\partial ( \partial p'_{n+1} - x) = \partial e_n (q_n)$,
this is equal to $-\partial x = -(-1)^{n+1}e_n(y)$ by definition. Then, $ e_n \partial q_n=\partial e_n( q_n)  = - \partial x = -(-1)^{n+1}e_n(y)$. Since $e_n$ is injective, this implies that $\partial q_n = -(-1)^{n+1} y$. Let $q_n' = - (-1)^{n+1} q_n$, so that $\partial q_n' = y$.

So far, we found $p'_{n+1} \in \bigoplus\limits_{v \in N_{\mathscr{V}}} C_{n+1}(\mathscr{R}_v)$ and $q'_n \in \bigoplus\limits_{e \in N_{\mathscr{V}}} C_n(\mathscr{R}_e)$ such that
$\partial q'_n = y$  and
\[ x= \partial p'_{n+1} -e_n(q_n) = \partial p'_{n+1} - e_n ( - (-1)^{n+1} q'_n) = \partial p'_{n+1} +(-1)^{n+1} e_n q'_n.\]
Thus, $[x,y]=0$ in $H_n(\total)$, and $\Phi_{\total}$ is injective.
\end{proof}

\section{Details of the proof of Lemma \ref{Theorem1}}
\label{ProofTheorem1}

\begin{Lemma}
The map $\Phi^i : H_0(C_{\bullet} \cosheaf{F}^i_n) \oplus H_1(C_{\bullet} \cosheaf{F}^i_{n-1}) \to H_n(\total^i)$ defined in Equation (\ref{Phi01}) is well-defined and bijective.
\end{Lemma}

\begin{proof}
We first show that $\Phi^i$ is well-defined. Assume that $( \{ \langle x \rangle \}, \{ \langle y \rangle \} ) = (\{ \langle x' \rangle \} , \{ \langle y' \rangle \})$ in $H_0(C_{\bullet} \cosheaf{F}^i_n) \oplus H_1(C_{\bullet} \cosheaf{F}^i_{n-1})$, i.e., $\{ \langle x \rangle \} = \{ \langle x' \rangle \}$ in $H_0(C_{\bullet} \cosheaf{F}^i_n)$ and $\{ \langle y \rangle \} = \{ \langle y' \rangle \}$ in $H_1(C_{\bullet} \cosheaf{F}^i_{n-1})$.

Note that
\[ \Phi^i_0 ( \{ \langle x \rangle \}) - \Phi^i_0 ( \{ \langle x' \rangle \}) = [ x- x', 0].\]
Since $\{ \langle x \rangle \} = \{ \langle x' \rangle \}$ in $H_0(C_{\bullet} \cosheaf{F}^i_n)$, there exists $p_{n+1} \in \bigoplus\limits_{v \in N_{\mathscr{V}}} C_{n+1}(\mathscr{R}^i_v)$ such that $\partial p_{n+1}= x- x'$. Thus, $[x-x',0]$ is trivial in $H_n(\total)$, and $\Phi^i_0$ is a well-defined map.

By construction, $\Phi^i_1( \{ \langle y \rangle \}) =  \Phi^i_1( \{ \langle y' \rangle \} )$. Thus, $\Phi^i$ is a well-defined map.

We now show that $\Phi^i$ is surjective. Given $[x, y] \in H_n(\total^i)$, we know that $\partial y=0$, so $\{ \langle y \rangle \}$ is an element of $H_1(C_{\bullet} \cosheaf{F}^i_{n-1})$. If $\mathscr{B}=\{ \, \{ \langle b_1 \rangle \}, \dots, \{ \langle b_t \rangle \} \, \}$ is a basis of $H_1(C_{\bullet} \cosheaf{F}^i_{n-1})$, and $b_1^*, \dots, b_t^*$ are the coset representatives of the basis, then $\{ \langle y \rangle \}$ can be written as
\[ \{ \langle y \rangle \} = c_1 \{ \langle b_1^* \rangle \} + \cdots + c_t \{ \langle b_t^* \rangle \} \]
for some $c_1, \dots, c_t$. That is, there exists $q_n \in \bigoplus\limits_{e \in N_{\mathscr{V}}} C_n (\mathscr{R}^i_e)$ such that
\begin{equation}
\label{Rewrite}
c_1 b_1^*+ \cdots + c_t b_t^* - y = \partial q_n.
\end{equation}

Recall from Equation (\ref{KeyEq}) that
\begin{equation}
\label{BetaB}
\partial \Gamma^i \{ \langle b \rangle \} = e^i_{n-1} (b^*)
\end{equation}
for every $\{ \langle b \rangle \} \in \mathscr{B}$. Let
\begin{equation}
\label{RnDef}
 r_n =  x-(-1)^{n+1}( c_1 \Gamma^i \{ \langle b_1^* \rangle \} + \cdots + c_t \Gamma^i \{ \langle b_t^* \rangle \})+(-1)^{n+1}e^i_n (q_n).
 \end{equation}
We know that $\partial x = (-1)^{n-1} e^i_{n-1}(y)$, and, by commutativity of Diagram \ref{SS0}, 
we have $\partial e^i_n(q_n) = e^i_{n-1} (\partial q_n)$. Thus, it follows that
\begin{align*}
\partial r_n &=  \partial x - (-1)^{n+1} \partial ( c_1 \Gamma^i \{ \langle b_1^* \rangle \} + \cdots + c_t \Gamma^i \{ \langle b_t^* \rangle \}) +(-1)^{n+1} \partial e^i_n (q_n) \\
&= (-1)^{n-1} e^i_{n-1}(y) - (-1)^{n+1} e^i_{n-1} (c_1 b_1^* + \cdots + c_t b_t^*) + (-1)^{n+1} e^i_{n-1} ( \partial q_n) \\
&=0.
\end{align*}
The second equality follows from Equation 
(\ref{BetaB}) and Diagram \ref{SS0}.
The third equality follows from Equation (\ref{Rewrite}). Thus, $\{ \langle r_n \rangle \}$ represents an element of $H_0(C_{\bullet} \cosheaf{F}^i_n)$. Then,
\begin{align*}
\Phi^i( \{ \langle r_n \rangle \} ,\{ \langle y \rangle \}) &= \Phi^i_0 ( \{ \langle r_n \rangle \} ) + \Phi^i_1 ( \{ \langle y \rangle \} )\\
&= [r_n +(-1)^{n+1}( c_1 \Gamma^i \{ \langle b_1^* \rangle \} + \cdots + c_t \Gamma^i \{ \langle b_t^* \rangle \}), c_1 b_1^*+ \dots + c_t b_t^* ] \\
&=[x+(-1)^{n+1} e^i_n(q_n), y+ \partial q_n] \\
&=[x,y] + [(-1)^{n+1}e^i_n(q_n), \partial q_n] \\
&=[x,y]
\end{align*}
The third equality follows from Equations (\ref{Rewrite}) and (\ref{RnDef}). Thus, $\Phi^i$ is surjective.

Lastly, we show that $\Phi^i$ is injective. Let $( \{ \langle x \rangle \}, \{ \langle y \rangle \} ) \in H_0(C_{\bullet} \cosheaf{F}^i_n) \oplus H_1(C_{\bullet} \cosheaf{F}^i_{n-1})$. If $\mathscr{B}=\{ \, \{ \langle b_1 \rangle \}, \dots, \{ \langle b_t \rangle \} \, \}$ is a basis of $H_1(C_{\bullet} \cosheaf{F}^i_{n-1})$, and $b_1^*, \dots, b_t^*$ are the coset representatives of the basis, then $\{ \langle y \rangle \}$ can be written as
\[ \{ \langle y \rangle \} = c_1 \{ \langle b_1^* \rangle \} + \cdots + c_t \{ \langle b_t^* \rangle \} \]
for some $c_1, \dots, c_t$. Assume that
\[\Phi^i (\{ \langle x \rangle \}, \{ \langle y \rangle \} ) = [x+(-1)^{n+1}( c_1 \Gamma^i \{ \langle b_1^* \rangle \} + \cdots + c_t \Gamma^i \{ \langle b_t^* \rangle \}), c_1 b_1^*+ \dots + c_t b_t^*]=0.\]
Then, there exists $q_n \in \bigoplus\limits_{e \in N_{\mathscr{V}}} C_n(\mathscr{R}^i_e)$ and $p_{n+1} \in \bigoplus\limits_{v \in N_{\mathscr{V}}} C_{n+1}(\mathscr{R}^i_v)$ such that
\begin{equation}
\label{trivial1}
\partial q_n =c_1 b_1^*+ \dots + c_t b_t^*,
\end{equation}
\begin{equation}
\label{trivial2} 
\partial p_{n+1} + (-1)^{n-1}e^i_n (q_n)  =x+(-1)^{n+1}( c_1 \Gamma^i \{ \langle b_1^* \rangle \} + \cdots + c_t \Gamma^i \{ \langle b_t^* \rangle \}).\end{equation}

From Equation (\ref{trivial1}), we know $c_1 \{ \langle b_1^* \rangle \} + \dots + c_t \{ \langle b_t^* \rangle \} = \{ \langle y \rangle \}$ is trivial in $H_1(C_{\bullet} \cosheaf{F}^i_{n-1})$. Thus, $\Phi^i(\{ \langle x \rangle \} , \{ \langle y \rangle \}) = \Phi^i ( \{ \langle x \rangle \},0) = [x,0]$.

If $[x,0]$ is trivial in $H_n(\total^i)$, then there exists $a_n \in \bigoplus\limits_{e \in N_{\mathscr{V}}} C_n(\mathscr{R}^i_e)$ and $b_{n+1} \in \bigoplus\limits_{v \in N_{\mathscr{V}}} C_{n+1}(\mathscr{R}^i_v)$ such that
\[ \partial a_n =0, \]
\[ \partial b_{n+1} + (-1)^{n-1}e^i_n a_n  =x. \]
The above two equations imply that $\{ \langle x \rangle \}$ is trivial in $H_0(C_{\bullet} \cosheaf{F}^i_n)$ as well. Thus, $\Phi^i$ is injective.

\end{proof}

\end{appendices}

\bibliographystyle{abbr}

\end{document}